\documentclass[12pt]{amsart} 
\usepackage{fullpage}
\usepackage{amsmath,amssymb}

\usepackage{graphicx} 
\usepackage{graphicx} 
\usepackage{amssymb} 
\usepackage{epstopdf} 

\usepackage[all,cmtip]{xy} 

\theoremstyle{plain} 
\newtheorem{thm}{Theorem} 
\newtheorem{cor}[thm]{Corollary} 
\newtheorem{lem}[thm]{Lemma} 
\newtheorem{conj}[thm]{Conjecture} 
\newtheorem{prop}[thm]{Proposition} 
\theoremstyle{definition} 
\newtheorem{defn}[thm]{Definition} 
\newtheorem{rem}[thm]{Remark} 
\newtheorem{ex}[thm]{Example} 
\newtheorem{exs}[thm]{Examples} 

\newcommand{\iinfty}{{\mathchoice
{\begin{minipage}{.15in}\includegraphics[width=.15in]{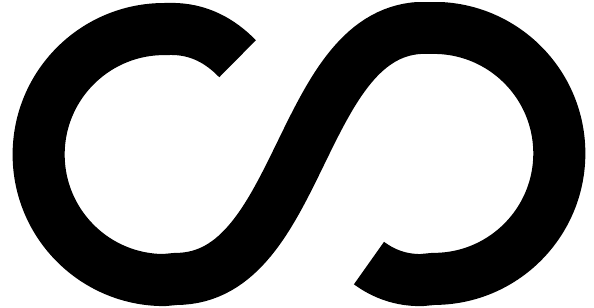}\end{minipage}}
{\begin{minipage}{.12in}\includegraphics[width=.12in]{infty2.pdf}\end{minipage}}
{\begin{minipage}{.10in}\includegraphics[width=.10in]{infty2.pdf}\end{minipage}}
{\begin{minipage}{.08in}\includegraphics[width=.08in]{infty2.pdf}\end{minipage}}
}}

\renewcommand{\int}{\operatorname{int}} 
 
\newcommand{\Ker}{\operatorname{Ker}} 

\newcommand{\id}{\operatorname{id}}

\newcommand{\Arf}{\mathrm{Arf}}

\newcommand{\im}{\operatorname{Im}}

\newcommand\W{\text{\sf W}} 
\newcommand\G{\text{\sf G}} 

\newcommand\sD{\text{\sf D}}
\newcommand\sL{\text{\sf L}}

\newcommand{\Z}{\mathbb{Z}} 
\newcommand{\N}{\mathbb{N}}

\newcommand{\bG}{\mathbb{G}} 
\newcommand{\bW}{\mathbb{W}} 
\newcommand{\bL}{\mathbb{L}} 
\newcommand{\bT}{\mathbb{T}}

\newcommand{\g}{\mathfrak{g}}

\newcommand{\cT}{\mathcal{T}}

\newcommand{\sra}{\twoheadrightarrow}

\newcommand{\QF}{\text{\sf{QF}}}
\newcommand{\CQF}{\text{\sf{CQF}}}
\newcommand{\SQF}{\text{\sf{SQF}}}
\newcommand{\SF}{\text{\sf{SF}}}
\newcommand{\HF}{\text{\sf{HF}}}
\newcommand{\QR}{\text{\sf{QR}}}

\newcommand{\SL}{\operatorname{SL}}
\newcommand{\z}{\mathbb Z_2}
\newcommand{\sK}{\text{\sf K}}

\begin{document} 

\title{Universal quadratic forms and untwisting Whitney towers} 

\begin{abstract} 
The first part of this paper completes the classification of Whitney towers in the $4$--ball that was started in three related papers. We provide an algebraic framework allowing the computations of the graded groups associated to geometric filtrations of classical link concordance by order $n$ (twisted) Whitney towers in the $4$--ball. Higher-order Sato-Levine invariants and higher-order Arf invariants are defined and shown to be the obstructions to framing a twisted Whitney tower. 

In the second part of this paper, a general theory of quadratic forms is developed and then specialized from the
non-commutative to the commutative to finally, the symmetric settings. The intersection invariant
for twisted Whitney towers is shown to be the universal symmetric refinement of the framed intersection invariant.

UPDATE: The results of the first six sections of this paper have been subsumed into the paper \emph{Whitney tower concordance of classical links} \cite{WTCCL}.
\end{abstract}

\author[J. Conant]{James Conant} 
\email{jconant@math.utk.edu} 
\address{Dept. of Mathematics, University of Tennessee, Knoxville, TN} 

\author[R. Schneiderman]{Rob Schneiderman} 
\email{robert.schneiderman@lehman.cuny.edu} 
\address{Dept. of Mathematics and Computer Science, Lehman College, City University of New York, Bronx, NY} 

\author[P. Teichner]{Peter Teichner} 
\email{teichner@mac.com} 
\address{Dept. of Mathematics, University of California, Berkeley, CA and} 
\address{Max-Planck Institut f\"ur Mathematik, Bonn, Germany}


\maketitle 
\section{Introduction}
In \cite{CST1,CST2} we examine three filtrations of link concordance: $\bG_n$ is the set of framed links in the $3$--sphere bounding disjointly embedded class $n+1$ gropes in the $4$--ball; $\bW_n$ is the set of framed links in the $3$--sphere bounding order $n$ framed Whitney towers in the $4$--ball; and $\bW_n^\iinfty$ is the set of framed links in the $3$--sphere bounding order $n$ twisted Whitney towers in the $4$--ball. Using the intersection theory of Whitney towers, we show in \cite{CST1} that the associated graded sets 
$\G_n, \W_n$, and $\W_n^\iinfty$ are finitely generated groups under a well-defined connected sum operation. Throughout this paper we fix (and usually suppress from notation) the number $m$ of link components; and as we remark in \cite{CST1}, the filtrations $\G_n$ and $\W_n$ are the same by \cite{S1}, so we will restrict to discussing $\W_n$ and $\W^\iinfty_n$. 

Whitney towers are built from iterated `layers' of Whitney disks on immersed surfaces in $4$-manifolds, motivated by the notion that higher-order Whitney towers are `better approximations' to successful Whitney moves (which would eliminate pairs of intersections).  
Intersections which are not paired by Whitney disks determine the intersection invariants of Whitney towers, which take values in abelian groups generated by labeled oriented unitrivalent trees, modulo antisymmetry, Jacobi, and some other relations. 

The vanishing of the intersection invariant for a Whitney tower of order $n$ implies the existence of an order $(n+1)$ Whitney tower. Non-trivial higher-order intersections can represent obstructions to finding embedded underlying surfaces, providing a `measure' of the general failure of the Whitney move in dimension four. A successful Whitney move requires
a framed Whitney disk, and the intersection theory of \emph{twisted} Whitney towers succinctly incorporates the framing obstructions into the general obstruction theory by assigning
certain ``twisted'' $\iinfty$-trees to \emph{twisted} Whitney disks, and introducing appropriate relations in the target groups.
As we show below in Section~\ref{sec:invt-forms-quadratic-refinements}, the twisted theory
can be interpreted as a universal quadratic refinement of the framed theory.   

In \cite{CST1}, we defined surjective \emph{realization maps}
\[
\widetilde{R}_{n}\colon \widetilde{\cT}_{n}\twoheadrightarrow \W_{n}  \quad \text{ and } \quad R^\iinfty_{n}\colon\cT^\iinfty_n\twoheadrightarrow \W^\iinfty_n
\]
where the abelian groups $\widetilde{\cT}_{n}$ are the targets for the intersection invariants of (framed) Whitney towers, and the abelian groups $\cT^\iinfty_n$ are the targets for the intersection invariants of twisted Whitney towers, see Definition~\ref{def:tree-groups}. Both are finitely generated by certain trivalent trees and hence upper bounds on the sizes of $\W_n$ and $\W_n^\iinfty$ are obtained. On the other hand, we show in \cite{CST2} how Milnor invariants induce maps $\mu_n$, which fit into a commutative triangle
$$\xymatrix{
\cT^\iinfty_n\ar@{->>}^{R_{n}^\iinfty}[r]\ar@{->>}_{\eta_n}[dr]&{\sf W}^\iinfty_n\ar@{->>}[d]^{\mu_n}\\
&{\sD}_n
}$$
where ${\sD}_n$ is a free abelian group determined by the image of the first non-vanishing order $n$ (length $n+2$) Milnor invariants, see Section~\ref{sec:main-defs}. Moreover,  $\eta_{n}$ is a combinatorially defined map that takes unrooted trees representing
higher-order intersections to rooted trees which represent iterated commutators determined by the link longitudes. 
In \cite{CST3} we prove a conjecture formulated by J. Levine in his study of $3$-dimensional homology
cylinders which implies that $\eta_n$ is an isomorphism for $n\equiv 0,1,3\mod 4$ (Theorem~\ref{thm:isomorphisms}). 

In this paper, the fourth in the series, we will show how the geometric computations from \cite{CST1,CST2} and the resolution of the Levine conjecture \cite{CST3} can be leveraged, using only algebra, to a complete classification of both $\W_n$ and $\W_n^\iinfty$ for $n\equiv 0,1,3\mod 4$. In the remaining orders, we complete the classification up to a well-controlled 2-torsion group. We refer to \cite{CST0} for an overview of this program, including applications to string links and $3$-dimensional homology cylinders \cite{CST5}.

\subsection{The twisted Whitney tower filtration}
Via the usual correspondence between rooted oriented labeled unitrivalent trees and non-associative brackets, the map $\eta_n$ (Definition~\ref{def:eta-maps-and-D-tilde}) essentially sums over all choices of placing roots at univalent vertices of order
$n$ trees (having $n$ trivalent vertices). It turns out to determine an element in the kernel $\sD_n$ of the bracket map $\sL_1\otimes\sL_{n+1}\to\sL_{n+2}$, where $\sL_k$ is the degree $k$ part of the free $\Z$-Lie algebra (on $m$ free generators).
As we remarked above, the fact that $\eta_n$ is an isomorphism in three quarters of the cases immediately implies:
\begin{thm}\label{thm:R-013-isomorphisms}
For $n\equiv 0,1,3\mod 4$, the maps $R_n^\iinfty$ and $\mu_n$ induce isomorphisms
$$\cT_n^\iinfty\cong \W^\iinfty_n\cong \sD_n$$.
\end{thm}
We show in Theorem~\ref{thm:isomorphisms} below that the kernel of $\eta_{4k-2}$ is isomorphic to $\z\otimes\sL_k$. This motivates the introduction of a larger group $\sD^\iinfty_{4k-2}$ (Definition~\ref{def:D-infty}), to which $\eta_{4k-2}$ lifts as an isomorphism:
\[
\xymatrix{
&&\sD^\iinfty_{4k-2}\ar@{->>}[d]&&\\
&&\sD_{4k-2} &&\\
\cT^\iinfty_{4k-2}\ar@{>->>}[uurr]^{\eta^\iinfty_{4k-2}}\ar@{->>}[rrrr]_{R^\iinfty_{4k-2}}\ar@{->>}[urr]_{\eta_{4k-2}}&&&&\W^\iinfty_{4k-2}\ar@{->>}[ull]^{\mu_{4k-2}}\ar@{>-->>}[uull]_{\mu^\iinfty_{4k-2}}
}
\]
If we could lift $\mu_{4k-2}$ to an invariant taking values in $\sD^\iinfty_{4k-2}$, then the last case would be solved: $\cT^\iinfty_{4k-2}\cong \W^\iinfty_{4k-2}\cong \sD^\iinfty_{4k-2}$. However, at present the existence of such a lift is only known for $k=1$. 
So the most we can say is:
\begin{thm}\label{thm:twistedalpha}
There is an exact sequence (which is short exact for $k=1$):
$$\z\otimes\sL_{k}\overset{\alpha_k}{\longrightarrow}\W^\iinfty_{4k-2}\overset{\mu_{4k-2}}{\longrightarrow}\sD_{4k-2}\to 0.$$ 
Moreover, $\alpha_k$ is injective if and only if the lift $\mu^\iinfty_{4k-2}$ exists. 
\end{thm}

\begin{conj}\label{conj:alpha-injective}
The homomorphisms $\alpha_k$ are injective for all $k\geq 1$. 
\end{conj}

Whether or not Conjecture~\ref{conj:alpha-injective} holds, we use Theorem~\ref{thm:twistedalpha} to define \emph{higher-order Arf invariants}
as ``inverses'' to the $\alpha_k$:
\begin{defn}\label{def:Arf-k}
For $\sK^{\iinfty}_{4k-2}:=\Ker\mu_{4k-2}$, define the map \emph{$\Arf_k$} by
$$\Arf_k\colon \sK^{\iinfty}_{4k-2}\to (\z\otimes\sL_k)/\Ker(\alpha_k)$$
\end{defn}
The terminology comes from the fact that for $k=1$, this is related to the classical Arf invariant of a knot. In fact, for a link $L$ our invariant $\Arf_1(L)$ contains the same information as the classical Arf invariants of its components \cite{CST2}. 
 Then we have the corollary:
\begin{cor}\label{cor:twisted-classification}
The groups $\W^\iinfty_{n}$ are classified by Milnor invariants $\mu_n$ and the above Arf invariants $\Arf_k$ for $n=4k-2$.
\end{cor}

So modulo the problem of determining the exact 2-torsion group which is the range of the higher-order Arf invariants, the twisted groups $\W^\iinfty_n$ are completely classified.

\subsection{The framed Whitney tower filtration}
It turns out that the analysis of the groups associated to the framed Whitney tower filtration follows by purely algebraic arguments from the twisted case, using the exact sequence
\begin{equation}\tag{\W}\label{4tseq}
0\to{\sf W}_{2n}\to{\sf W}^\iinfty_{2n}\to{\sf W}_{2n-1}\to{\sf W}^\iinfty_{2n-1}\to 0
\end{equation}
from \cite{CST1}, the resolution of Levine conjecture \cite{CST3}, and the following theorem, which will be proved in Section~\ref{sec:proof-tau-sequences}:
\begin{thm}\label{thm:Tau-sequences}
For any $n\in\N$, there are short exact sequences
\[
\xymatrix{ 
0\ar[r] & \cT_{2n} \ar[r] &  \cT^\iinfty_{2n} \ar[r] & \z \otimes \sL'_{n+1} \ar[r] & 0\\
 0\ar[r] &  \z \otimes \sL'_{n+1} \ar[r] & \widetilde\cT_{2n-1}  \ar[r] &  \cT^\iinfty_{2n-1} \ar[r] & 0
 } \]
\end{thm}
Here $\sL_k'$ denotes the degree $k$ part of the free \emph{quasi-Lie algebra}, which differs from the free Lie algebra (over $\Z$) only by replacing the usual self-annihilation relation $[X,X]=0$ with the weaker skew-symmetry relation $[X,Y]= - [Y,X]$. The various tree groups $\cT_n$ etc.~will be defined in Section~\ref{sec:main-defs}. 

The proof of the injectivity of $\cT_{2n}\hookrightarrow \cT^\iinfty_{2n}$ turns out to be the trickiest aspect of Theorem~\ref{thm:Tau-sequences}. There is a bilinear form $\langle\cdot,\cdot\rangle\colon\sL'_{n+1}\otimes\sL'_{n+1}\to \cT_{2n}$ defined by gluing the roots of two trees together and erasing the resulting bivalent vertex. This bilinear form corresponds to the intersection pairing of Whitney disks in an order $2n$ Whitney tower \cite{CST1}.
The map 
\[
\sL'_{n+1}\to \cT^\iinfty_{2n}, \quad J\mapsto J^\iinfty
\]
 corresponds to the association of ``twisted'' rooted trees to twisted Whitney disks (\cite{CST1} and Definition~\ref{def:tree-groups} below), and is a quadratic refinement of the bilinear form $\langle\cdot,\cdot\rangle$, in the sense that 
 \[
 (J_1+J_2)^\iinfty=J_1^\iinfty+J_2^\iinfty+\langle J_1,J_2\rangle.
 \]
In Section~\ref{sec:invt-forms-quadratic-refinements} we will present a novel construction of universal quadratic refinements, which when specialized to the  symmetric bilinear form $\langle\cdot,\cdot\rangle$, shows that $\cT^\iinfty_{2n}$ is in fact the \emph{universal} quadratic refinement. As a direct consequence of this general framework, we will get the injectivity of $\cT_{2n}\hookrightarrow \cT^\iinfty_{2n}$.

In order to proceed with the analysis of the framed case, notice that
the group 
\[
\sK^\mu_{2n-1}:=\Ker(\W_{2n-1} \sra \W^\iinfty_{2n-1}\cong\sD_{2n-1})
\]
 breaks the $4$-term exact sequence (\ref{4tseq}) into two short exact sequences.  Calculating $\sK^\mu_{2n-1}$ will thus allow us to compute $\W_{2n}$ and $\W_{2n-1}$ in terms of the corresponding $\iinfty$-groups.
On the other hand, the group $\sK^\mu_{2n-1}\cong\bW^\iinfty_{2n}/\bW_{2n}$ precisely measures the obstructions to framing a Whitney tower of order $2n$. 
It turns out that much of this group  ${\sK}^\mu_{2n-1}$ can be detected by \emph{higher-order Sato-Levine invariants}:

\begin{thm}\label{thm:SL-invariants}
There are  epimorphisms $\SL_{2n-1}\colon{\sf K}^\mu_{2n-1}\sra \z\otimes\sL_{n+1}$ which we shall refer to as \emph{order $2n-1$ Sato-Levine invariants.} These are isomorphisms for even $n$.
\end{thm}
These invariants are defined as a certain projection of Milnor invariants of order $2n$ (see Section~\ref{sec:framed-classification}) and they are generalizations of (the mod 2 reduction of) an invariant defined by Sato \cite{Sa} and Levine for $n=1$. 

The theorem above gives a complete computation when $n$ is even. When $n$ is odd, consider the kernel 
\[
\sK^{\SL}_{4k-3}:=\Ker(\SL_{4k-3}\colon \sK^\mu_{4k-3}\to \z\otimes\sL_{2k})
\]
 Then the unknown part of the framed filtration can be identified with the unknown part of the twisted filtration:
\begin{prop}\label{prop:canonical-iso}
The groups ${\sf K}^{\SL}_{4k-3}$ and ${\sf K}^{\iinfty}_{4k-2}$ are canonically isomorphic.
\end{prop}

From this we get a complete obstruction theory for the framed filtration:

\begin{cor}
The associated graded groups $\W_n$ are classified by Milnor invariants $\mu_n$ and in addition, Sato-Levine invariants $\SL_n$ if $n$ is odd, and finally, Arf invariants $\Arf_k$ for $n=4k-3$.
\end{cor}

The missing piece in this story is a precise determination of $\sK_k:=\sK^\iinfty_{4k-2}\cong \sK^{\SL}_{4k-3}$. In Theorem~\ref{thm:isomorphisms} below, $\Ker(\eta_{4k-2})$ is identified with
$\z\otimes{\sf L}_k$ via 
$(J,J)^\iinfty\leftrightarrow 1\otimes J$, and links in $\W_{4k-2}$ realizing the elements $(J,J)^\iinfty$ are 
constructed in \cite{CST2} as iterated band sums of any knot with non-trivial Classical Arf invariant.
Since these links span the image of $\alpha_k$, 
the remaining geometric question is:
Which, if any, of these links  actually bound order $4k-1$ twisted
Whitney towers?
Conjecture~\ref{conj:alpha-injective} states that none of them do.

\subsection{The Master Diagrams}
The preceding discussion can be organized into ``master diagrams" that contain the content of the classification story. There will be two such diagrams, each covering half of the cases,  Section~\ref{sec:main-defs} contains all definitions. 
The first diagram is the simpler:
\[
\xymatrix{ 
 \cT_{4k}\ar@{>->>}[dd]_(.6){\eta_{4k}'} \ar@{>->>}[dr]^{\widetilde{R}_{4k}} \ar@{>->}[rr] & &  \cT^\iinfty_{4k} \ar@{>->>}[dd]_(.6){\eta^\iinfty_{4k}}  \ar@{>->>}[dr]^{R^\iinfty_{4k}} \ar[rr] &&\widetilde{\cT}_{4k-1} \ar@{>->>}[dd]_(.6){\widetilde{\eta}_{4k-1}}  \ar@{>->>}[dr]^{\widetilde{R}_{4k-1}} \ar@{->>}[rr] &&\cT^\iinfty_{4k-1}\ar@{>->>}[dd]_(.6){\eta^\iinfty_{4k-1}}\ar@{>->>}[dr]^{R^\iinfty_{4k-1}}&
 \\
  & \W_{4k} \ar@{>->>}[dl]^{\mu_{4k}} \ar@{>->}[rr]|(.50)\ &&  \W^\iinfty_{4k}\ar @{>->>}[dl]^{\mu_{4k}}   \ar[rr]|(.485)\ && \W_{4k-1}\ar @{>->>}[dl]^{\widetilde{\mu}_{4k-1}}   \ar@{->>}[rr]|(.485)\ && \W^\iinfty_{4k-1}\ar@{>->>}[dl]^{\mu_{4k-1}} 
  \\
\sD'_{4k} \ar@{>->}[rr] &&  \sD_{4k} \ar[rr]&&  \widetilde{\sD}_{4k-1} \ar@{->>}[rr]&&\sD_{4k-1}&
}
\]
Two of the three 4-term horizontal sequences are exact, by equation (\ref{4tseq}) above and Theorem~\ref{thm:Tau-sequences}. The four vertical $\eta$-maps are isomorphisms by Theorem~\ref{thm:isomorphisms}, which then implies that the $R$-maps and $\mu$-maps are all isomorphisms, using the surjectivity of the $R$-maps. Exactness of the horizontal $\sD$-sequence also follows, with $\sD'$ denoting the quasi-Lie bracket kernel, and $\widetilde{\sD}_{4k-1}$ the quotient of 
$\sD'_{4k-1}$ by the image under $\eta'$ of the relations in $\widetilde{\cT}_{4k-1}$. The isomorphism $\widetilde{\mu}_{4k-1}$ is the unique map making the diagram commute.

Moreover, the three horizontal sequences can each be split into two short exact sequences, with the term $\sK^\mu_{4k-1}\cong \z\otimes\sL_{2k+1}$ appearing in the middle as the cokernel of the left-hand maps and the kernel of the right-hand maps. 

The other half of cases are covered by the following diagram:
\[
\xymatrix{ 
 \cT_{4k-2}\ar@{>->>}[dd]_(.6){\eta_{4k-2}}  \ar@{>->>}[dr]^{\widetilde{R}_{4k-2}} \ar@{>->}[rr] &&  \cT^\iinfty_{4k-2} \ar@{>->>}[dd]_(.6){\eta^\iinfty_{4k-2}}  \ar@{->>}[dr]^{R^\iinfty_{4k-2}} \ar[rr] && \widetilde\cT_{4k-3} \ar@{>->>}[dd]_(.6){\widetilde\eta_{4k-3}} \ar@{->>}[dr]^{\widetilde{R}_{4k-3}} \ar@{->>}[rr] &&  \cT^\iinfty_{4k-3}\ar@{>->>}[dd]_(.6){\eta^\iinfty_{4k-3}} \ar@{>->>}[dr]^{R^\iinfty_{4k-3}}\\
& \W_{4k-2} \ar@{>->>}[dl]^{\mu_{4k-2}} \ar@{>->}[rr]|(.48)\ &&  \W^\iinfty_{4k-2}\ar @{->>}[ddl]^(.3){\mu_{4k-2}}|(.5){\phantom{X}}|(.66){\phantom{X}}  \ar@{-->}[dl]|{\mu^\iinfty_{4k-2} ?} \ar[rr]|(.52)\ && \W_{4k-3}  \ar@{-->}[dl]|{\widetilde \mu_{4k-3} ?}  \ar@{->>}[rr]|\  &&  \W^\iinfty_{4k-3} \ar@{>->>}[dl]^{\mu_{4k-3}}\\
\sD'_{4k-2} \ar@{=}[d] \ar@{>->}[rr] &&  \sD^\iinfty_{4k-2} \ar@{->>}[d]_{p_{4k-2}} \ar@{->>}[dr]^{s\ell'_{4k-2}}\ar[rr] && \widetilde {\sD}_{4k-3}  \ar@{->>}[rr] &&  \sD_{4k-3}\\
\sD'_{4k-2} \ar@{>->}[rr]\ &&  \sD_{4k-2} \ar@{}[r]|(.42)*[F-,]{\, p.b.\,} \ar@{->>}[dr]_{s\ell_{4k-2}} & \sL'_{2k}\otimes\z \ar@{->>}[d]_p \ar@{>->}[ur]&&& \\
&&&  \sL_{2k} \otimes\z&&&
 } \]

As in the first diagram, the three 4-term horizontal sequences are exact. The horizontal $\cT$- and 
$\sD$-sequences actually break into two short exact sequences with the groups $\z\otimes \sL'_{2k}$ in the middle. For the horizontal $\W$-sequence, the group $\sK^\mu_{4k-3}$ sits in the middle, which we conjecture to be isomorphic to $\z\otimes\sL'_{2k}$. At the bottom, there are two short exact sequences, both starting with $\sD'_{4k-2}$ because the diagonal square is a pullback (p.b.). The existence of the lifts $\mu^\iinfty_{4k-2}$ is equivalent to the existence of the lifts $\widetilde{\mu}_{4k-3}$ and to Conjecture~\ref{conj:alpha-injective} above. These lifts would automatically be isomorphisms by the commutativity of the diagram.

{\bf Acknowledgments:} 
This paper was written while the first two authors were visiting   the third author at the Max-Planck-Institut f\"ur Mathematik in Bonn. They all thank MPIM for its stimulating research environment and generous support. The last author was also supported by NSF grants DMS-0806052 and DMS-0757312.

\section{The main definitions}\label{sec:main-defs}
This section contains definitions of the groups and maps necessary to prove the classification theorems for ${\sf W}_n$ and ${\sf W}^\iinfty_n$.
For details and definitions on (twisted) Whitney tower concordance and Milnor invariants see \cite{CST1,CST2}. 

\begin{defn}\label{def:Lie-and-quasi-Lie} Fix an integer $m\geq 1$. 
The abelian group ${\sf L}_n=\sL_n(m)$ is the degree $n$ part of the free Lie algebra over $\Z$ on $m$ generators 
$\{X_1,X_2,\ldots,X_m\}$. Similarly, $\sL'_n=\sL'_n(m)$ is the degree $n$ part of the free {\em quasi} Lie algebra which is obtained by replacing the usual self-annihilation relation $[X,X]=0$ in a Lie algebra with the weaker skew-symmetry relation $[X,Y]= - [Y,X]$.

The group ${\sf D}_n$ is defined as the kernel of the bracket map $\sL_1 \otimes \sL_{n+1}\to \sL_{n+2}$, and ${\sf D}'_n$ is the kernel of the corresponding quasi-Lie bracket map. 
\end{defn}
Note that there is a natural projection $p\colon{\sf L}'_n\twoheadrightarrow{\sf L}_n$. Levine \cite{L3} showed the kernel of this projection is $\z\otimes {\sf L}_k$ for $n=2k$ and is trivial for odd $n$.
\begin{defn}\label{def:sl-maps} The epimorphisms $s\ell_{2k}\colon {\sf D}_{2k}\sra \z\otimes{\sL}_{k+1}$ are defined by the snake lemma applied to the diagram:
$$\xymatrix{
&&\mathbb Z_2\otimes\sL_{k+1}\ar@{>->}[d]^{sq}\\
\sD'_{2k}\ar@{>->}[d] \ar@{>->}[r]&\sL_1\otimes\sL'_{2k+1}\ar@{>->>}[d]^\cong\ar[r]&\sL'_{2k+2}\ar@{->>}[d]\\
\sD_{2k} \ar@{>->}[r]\ar@{-->>}[d]^{s\ell_{2k}}&\sL_1\otimes \sL_{2k+1}\ar@{->>}[r]&\sL_{2k+2}\\
\z\otimes \sL_{k+1}&&
}$$
The two horizontal sequences are exact by definition and the vertical sequence on the right is exact by Theorem 2.2 of \cite{L3}. The squaring map on the upper right is $sq(1\otimes a) := [a,a]$.
\end{defn}

\begin{defn}\label{def:D-infty}
The groups $\sD^\iinfty_{4k-2}$ (as well as the maps $s\ell_{4k-2}',p_{4k-2},sq^\iinfty$) are defined by the pullback diagram
\begin{equation}\tag{$\sD^\iinfty$}
\xymatrix{
\z\otimes \sL_{k} \ar@{>->}[r]^{sq^\iinfty} \ar@{=}[d] & \sD^\iinfty_{4k-2}\ar@{->>}[d]^{s\ell_{4k-2}'}\ar@{->>}[r]^{p_{4k-2}}&\sD_{4k-2}\ar@{->>}[d]^{s\ell_{4k-2}}\\
\z\otimes \sL_{k} \ar@{>->}[r]^{sq} & \Z_2\otimes \sL'_{2k}\ar@{->>}[r]^p &\Z_2\otimes \sL_{2k}
}
\end{equation}
\end{defn}

All trees considered in this paper are \emph{unitrivalent}, and equipped with \emph{vertex orientations} (cyclic orderings of the edges incident to each trivalent vertex). Such trees are graded by \emph{order}, which is the number of trivalent vertices. The univalent vertices of a tree are
labeled by elements of the index set $\{1,2,\ldots,m\}$, except for the {\em root vertex} of a rooted tree which is usually left unlabeled.

Given rooted trees $I$ and $J$, the \emph{rooted product} $(I,J)$ is the rooted tree gotten by identifying the two roots to a vertex and adjoining a rooted edge to this new vertex, with the orientation of the new trivalent vertex given by the ordering of $I$ and $J$ in $(I,J)$. The inner product $\langle I,J \rangle$ of $I$ and $J$ is defined to be the unrooted tree gotten by identifying the two rooted edges to a single edge. 

We use the notational convention that rooted trees are written in caps like $J$ whereas we use lower case letters like $t$ for unrooted trees. 
\begin{defn}\label{def:tree-groups}  The various $\cT$-groups are defined as follows:
\begin{enumerate}
\item The abelian group $\mathcal T_n$ is generated by order $n$ (unrooted) trees, modulo AS and IHX relations
\begin{figure}[h]
\centerline{\includegraphics[scale=.65]{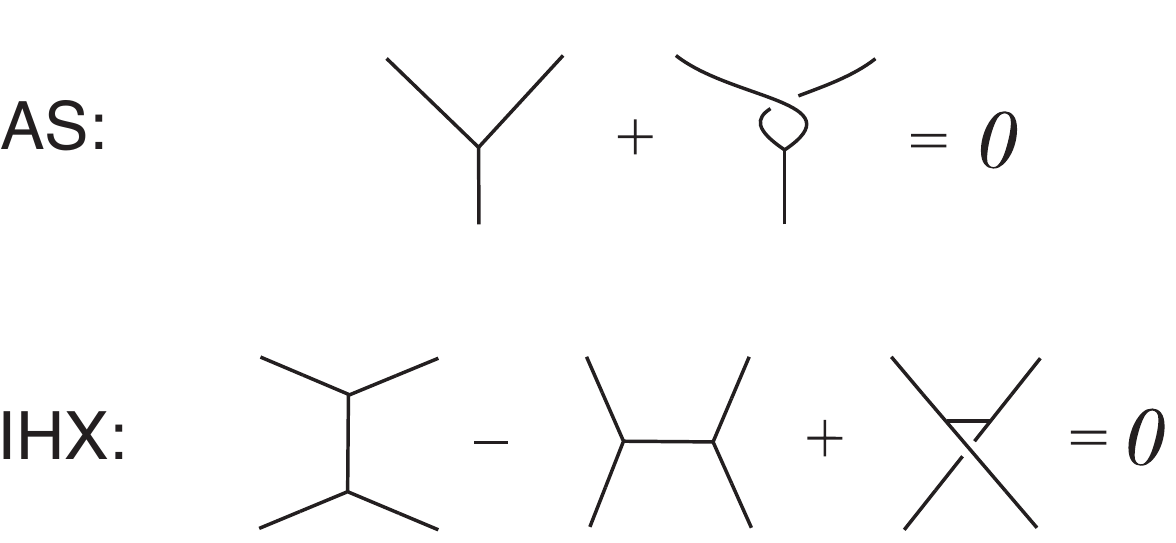}}
         \caption{Local pictures of the \emph{antisymmetry} (AS) and \emph{Jacobi} (IHX) relations
         in $\cT$. Here all trivalent orientations are induced from a fixed orientation of the plane, and univalent vertices possibly extend to subtrees which are fixed in each equation.}
         \label{fig:ASandIHXtree-relations}
\end{figure}
\item The group $\widetilde{\mathcal T}_{2n}$ is defined to be $\mathcal T_{2n}$, whereas $\widetilde{\mathcal T}_{2n-1}$ is
by definition the quotient of $\mathcal T_{2n-1}$ by the \emph{framing relations}. These framing relations are defined as the image of homomorphisms 
\[
\Delta_{2n-1}:\Z_2\otimes\cT_{n-1}\rightarrow \cT_{2n-1}
\]
which is defined for generators $t\in\cT_{n-1}$ by $\Delta (t):=\sum_{v\in t} \langle i(v),(T_v(t),T_v(t))\rangle$,
where the sum is over all univalent vertices $v$ of $t$, with $i(v)$ the original label of $v$, and $T_v(t)$ denotes the rooted tree gotten by replacing $v$ with a root.
\item 
 The group $\cT^{\iinfty}_{2n-1}$ is the quotient of $\widetilde{\cT}_{2n-1}$ by the \emph{boundary-twist relations}: 
\[
\langle  (i,J),J \rangle \,=\, i\,-\!\!\!\!\!-\!\!\!<^{\,J}_{\,J}\,\,=\,0
\] 
Here $J$ ranges over all order $n-1$ rooted trees (and the first equality is just a reminder of notation).
\item The group $\cT^{\iinfty}_{2n}$ is gotten from $\widetilde{\cT}_{2n}=\cT_{2n}$ by including order $n$ $\iinfty$-trees as new generators. These are rooted trees as above, except that the root carries the label $\iinfty$. In addition to the IHX- and AS-relations on unrooted trees, we introduce the 
following new relations:
$$
J^\iinfty=(-J)^\iinfty\quad\quad
 2\cdot J^\iinfty=\langle J,J\rangle 
 $$
 as well as the local twisted IHX-relations
 \[
  I^\iinfty=H^\iinfty+X^\iinfty- \langle H,X\rangle 
\]
 \end{enumerate}
\end{defn}
These new relations will be explained `algebraically' in our last section via the theory of universal quadratic refinements; see \cite{CST1} for their geometric meanings. 

\begin{rem}
By definition of the free quasi-Lie algebra, the groups $\sL'_{n+1}$ are generated by rooted trees of order $n$ and the only relations are antisymmetry and the Jacobi identity. These are equivalent to the AS- and IHX-relations above if these are applied locally in rooted trees. Moreover, the groups $\sL_1\otimes \sL'_{n+1}$ are presented exactly the same way, except that the root vertex in any generating tree has also a label from the set $\{1,\dots,m\}$. 
\end{rem}

\begin{defn}\label{def:eta-maps-and-D-tilde} The various $\eta$-maps are defined as follows:
\begin{enumerate}
\item As in \cite{L2}, the homomorphism $\eta'_n\colon \cT_n\to\sD'_n$ is defined by summing over all possible ways of adding a root to an unrooted tree. This can be described in terms of quasi-Lie algebras as follows: For $v$ a univalent vertex of an order $n$ tree
$t$, denote by $B_v(t)\in\sL'_{n+1}$ the Lie bracket of generators $X_1,X_2,\ldots,X_m$ determined by the formal bracketing of indices
which is gotten by considering $v$ to be a root of $t$. Then $\eta_n(t):=\sum_{v\in t} X_{\ell(v)}\otimes B_v(t)$, denoting the label of a univalent vertex $v$ by $\ell(v)\in\{1,2,\ldots,m\}$. It was shown in \cite{L2} that this element lies in $\sD'_n$, the kernel of the bracketing map. 

\item The homomorphism $\eta_n : \cT^\iinfty_n \to \sD_n$ is defined on generators as follows: For unrooted trees, we take the above formula for $\eta'_n$ but viewed in Lie algebras, rather than quasi-Lie algebras. For rooted trees $J$, we set:
$$
\eta_n(J^\iinfty):=\frac{1}{2}\eta_n(\langle J,J \rangle)
$$
This expression lies in $\sL_1 \otimes \sL_{n+1}$ because the coefficient of $\eta_n(\langle J,J \rangle)$ is even and the division by two is well-defined because this is a torsion-free group. In fact, $\eta_n$ is a surjection onto $\sD_n$ by Lemma~\ref{lem:eta-infty-well-defined}.
\item The group $\widetilde{\sf D}_{2k-1}$ is defined to be the quotient of ${\sf D}'_{2k-1}$ by $\eta'_{2k-1}(\im \Delta_{2k-1})$. Thus there is an induced map $\widetilde{\eta}_{2k-1}\colon \widetilde{\cT}_{2k-1}\to \widetilde{\sD}_{2k-1}$.
\end{enumerate}
\end{defn}
\begin{lem}\label{lem:eta-infty-well-defined}
The homomorphism $\eta_n:\cT^\iinfty_n\rightarrow \sD_n$ is well-defined.
\end{lem}
\begin{proof}
In \cite{L2} it was shown that the map $\eta'_n:\cT_n\to \sD'_n:=\Ker\{\sL'_1 \otimes \sL'_{n+1}\to\sL'_{n+2}\}$ is a well-defined surjection. It follows that $\eta_n$ vanishes on the usual IHX and AS relations, and maps into $\sD_n$. So it suffices to check that $\eta_n$ respects the other 
relations in $\cT^\iinfty_n$.

First consider the odd order case. To show $\eta_n$ is well defined, it suffices to observe that 
$\eta_n(\langle (i,J),J\rangle)=0$, since placing a root at the \mbox{$i$-labeled} vertex determines a trivial symmetric bracket in $\sL_{n+1}$,
and all the other Lie brackets come in canceling pairs corresponding to putting roots on vertices in the isomorphic
pair of $J$ sub-trees.

Now considering the even order case, we have
$$
\eta_n((-J)^\iinfty)=\frac{1}{2}\eta_n(\langle -J,-J\rangle)=\frac{1}{2}\eta_n(\langle J,J\rangle)=\eta_n(J^\iinfty)
$$
and
$$
\eta_n(2J^\iinfty)=2\cdot\eta_n(J^\iinfty)=2\cdot\frac{1}{2}\eta_n(\langle J,J\rangle)=\eta_n(\langle J,J\rangle).
$$
And for $I$, $H$, and $X$, the terms in a Jacobi relation $I=H-X$, we have
$$
\begin{array}{ccl}
\eta_n(I^\iinfty)  & =  & \frac{1}{2}\eta_n(\langle I,I\rangle)  \\
  & =  & \frac{1}{2}\eta_n(\langle H-X,H-X\rangle)  \\
  & =  &  \frac{1}{2}(\eta_n(\langle H,H\rangle)-2\cdot\eta_n(\langle H,X\rangle)+\eta_n(\langle X,X\rangle))\\
  & = & \eta_n(H^\iinfty+X^\iinfty-\langle H, X\rangle )
\end{array}
$$
where the second equality comes from applying the Jacobi relation to the sub-trees $I$ inside the inner product $\langle I,I\rangle$ and expanding.
\end{proof}

\begin{lem}\label{lem:lifted eta}
There is a lift $\eta_{4k-2}^\iinfty$ of $\eta_{4k-2}$ 
\[
\xymatrix{
& \sD^\iinfty_{4k-2}\ar@{->>}[d]^{p_{4k-2}}\\
\cT^\iinfty_{4k-2} \ar@{->>}[r]^{\eta_{4k-2}} \ar@{->>}[ru]^{\eta^\iinfty_{4k-2}} & \sD_{4k-2} 
}\]
such that $\eta_{4k-2}^\iinfty((J,J)^\iinfty) = sq^\iinfty(1\otimes J)$ for all rooted trees $J\in \sL_k$.
\end{lem}
\begin{proof}
To construct $\eta^\iinfty_{4k-2}$ it suffices to observe that we have a commutative diagram
$$
\xymatrix{\cT^\iinfty_{4k-2}\ar[r]^{\eta_{4k-2}}\ar[d]&{\sD}_{4k-2}\ar[d]\\
\z\otimes {\sL}'_{2k}\ar[r]^p&\z\otimes{\sL}_{2k}
}
$$
which gives rise to a map to the pullback $\sD_{4k-2}^\iinfty$ from  diagram ${\sf D}^\iinfty$ above. To calculate $\eta^\iinfty_{4k-2}((J,J)^\iinfty)$ notice that $\eta_{4k-2}((J,J)^\iinfty)=0$ and $s\ell^\prime_{4k-2}((J,J)^\iinfty)=1\otimes (J,J)=sq(1\otimes J)$. So $\eta^\iinfty_{4k-2}((J,J)^\iinfty)=sq^\iinfty(1\otimes J)$. 
\end{proof}

\begin{rem}
The superscripts in our $\eta$-maps reflect those of their {\em target} groups.
\end{rem}

\section{The main isomorphisms}\label{sec:main-isos}
\begin{thm}\label{thm:isomorphisms} The following maps are isomorphisms
\begin{enumerate}
\item $\eta'_n\colon \cT_n\to \sD'_n$
\item $\widetilde{\eta}_{2k-1}\colon\widetilde{\cT}_{2k-1}\to \widetilde{\sD}_{2k-1}$
\item $\eta_{2k-1}\colon\cT^\iinfty_{2k-1}\to\sD_{2k-1}$
\item $\eta_{4k}\colon\cT^\iinfty_{4k}\to\sD_{4k}$
\item $\Ker(\eta_{4k-2})\to \z\otimes{\sf L}_k$ where $(J,J)^\iinfty\mapsto 1\otimes J$
\item $\eta^\iinfty_{4k-2}\colon \cT^\iinfty_{4k-2}\to\sD^\iinfty_{4k-2}$
\end{enumerate}
\end{thm}

\begin{proof}
Part (i) is the Levine conjecture, proven in \cite{CST3}, and (ii) follows by definition. 

\noindent {\bf Proof of (iii)}:
Consider the following diagram (which is commutative by definition of $\eta$):
$$\xymatrix{
0\ar[r]&\operatorname{span}\{\langle i,(J,J)\rangle\}\ar[r]\ar[d]^{\eta'}&\mathcal T_{2k-1}\ar[r]\ar[d]_\cong^{\eta'}&\mathcal T_{2k-1}^\iinfty\ar[r]\ar[d]^{\eta}&0\\
0\ar[r]&\mathbb Z_2^m\otimes {\sf L}_k\ar[r]^{sq}&\sD'_{2k-1}\ar[r]&\sD_{2k-1}\ar[r]&0\\
}
$$
The bottom row is exact by Corollary 2.3 of \cite{L3}. The top row is exact by definition and the middle map $\eta'$ is an isomorphism by (i). This implies the left-hand restriction $\eta'$ is one to one, and it is onto since $\eta'(\langle i,(J,J)\rangle)=X_i\otimes J$. Therefore, the right-hand map $\eta$ is an isomorphism by the 5-lemma. 

\noindent {\bf Proof of (iv)}:
Consider the following diagram (which is commutative by Lemma~\ref{lem:cd} below)
$$\xymatrix{
0\ar[r]&\cT_{4k}\ar[r]\ar[d]_\cong^{\eta'_{4k}}&\cT_{4k}^\iinfty\ar[r]\ar[d]^{\eta_{4k}}&\z\otimes{\sL}'_{2k+1}\ar[r]\ar[d]^p_{\cong}&0\\
0\ar[r]&{\sD}'_{4k}\ar[r]&{\sD}_{4k}\ar[r]^{s\ell_{4k}\,\,\,\,\,\,\,\,\,\,\,\,}&\z\otimes{\sL}_{2k+1}\ar[r]&0
}
$$
The bottom horizontal sequence is exact by Definition~\ref{def:sl-maps}. The top one is a case of Theorem~\ref{thm:Tau-sequences}, proven in Section~\ref{sec:proof-tau-sequences}.
Since ${\sL}_{2k+1}\cong {\sL}'_{2k+1}$ it follows that $\eta_{4k}$ is an isomorphism.

\noindent {\bf Proof of (v)}: 
Using Theorem 2.2 of \cite{L3}, Lemma~\ref{lem:cd} below and the fact that $\Ker \eta'=0$, we have a commutative diagram of exact sequences:
$$
\xymatrix{
&\Ker\eta_{4k-2}\ar[r]^\cong\ar@{>->}[d]&\mathbb Z_2\otimes {\sf L}_{k}\ar@{>->}[d]^{sq}\\
 \mathcal T_{4k-2}\ar@{>->}[r]\ar[d]_\cong^{\eta'_{4k-2}}& \mathcal T^\iinfty_{4k-2}\ar@{->>}[r]\ar@{->>}[d]^{\eta_{4k-2}}& \mathbb Z_2\otimes{\sf L}'_{2k}\ar@{->>}[d]^p\\
\sD'_{4k-2}\ar@{>->}[r]& \sD_{4k-2}\ar@{->>}[r]^(.4){s\ell_{4k-2}}&\mathbb Z_2\otimes{\sf L}_{2k}
}
$$
Hence $\Ker\eta_{4k-2}\cong \z\otimes{\sL}_k$ as desired. It remains to show the isomorphism is given by $(J,J)^\iinfty\mapsto 1\otimes J$. Clearly $(J,J)^\iinfty\in \Ker \eta_{4k-2}$, and it maps to $1\otimes [J,J]=sq(1\otimes [J])$ under the map $\cT^\iinfty_{4k-2}\to \z\otimes \sL'_{2k}$. 

\noindent {\bf Proof of (vi)}:
Diagram $\sD^\iinfty$ and Lemma~\ref{lem:lifted eta} imply that we have a commutative diagram:
$$
\xymatrix{
0\ar[r]&\langle(J,J)^\iinfty\rangle\ar[r]\ar@{->>}[d]&\mathcal T^{\iinfty}_{4k-2}\ar[r]^{\eta_{4k-2}}\ar[d]_{\eta^\iinfty_{4k-2}}&{\sf D}_{4k-2}\ar[r]\ar@{=}[d]&0\\
0\ar[r]&\z\otimes{\sf L}_{k}\ar[r]^(.6){sq^\iinfty_{4k-2}}&{\sD}^\iinfty_{4k-2}\ar[r]^{p_{4k-2}}&{\sD}_{4k-2}\ar[r]&0
}
$$
where the vertical left hand map sends $(J,J)^\iinfty$ to $1\otimes J$.
The right-hand square commutes by the main commutative triangle of Lemma~\ref{lem:lifted eta}, whereas the left square commutes by the calculation $\eta^\iinfty_{4k-2}((J,J)^\iinfty)=sq^\iinfty_{4k-2}(1\otimes J)$, as in Lemma~\ref{lem:cd}. 
Thus $\eta^\iinfty_{4k-2}$ is an isomorphism by the 5-lemma.
\end{proof}

\begin{lem}\label{lem:cd}
Sending $J^\iinfty$ to $1\otimes J$ gives a commutative diagram
$$\xymatrix{
\cT_{2k}^\iinfty\ar[r]\ar[d]^{\eta_{2k}}&\z\otimes{\sL}'_{k+1}\ar[d]^p\\
{\sD}_{2k}\ar[r]^{s\ell_{2k}\,\,\,\,\,\,\,\,\,\,\,\,}&\z\otimes{\sL}_{k+1}
}
$$
\end{lem}
\begin{proof}
First, we need a better handle on the map $sl_{2k}\colon\sD_{2k}\to \mathbb Z_2\otimes \sL_{k+1}$. Let $Z\in\sD_{2k}$ and pick a lift $Z'\in \sL_1\otimes\sL'_{2k+1}$. Tracing through the snake lemma diagram in Definition~\ref{def:sl-maps}, one sees that the bracket of $Z'$ is a sum of commutators $[J_i,J_i]$, and that $sl_{2k}(Z)=\sum 1\otimes J_i$.

Consider a tree $t\in\mathcal T^\iinfty_{2k}$ which maps to zero in $\mathbb Z_2\otimes {\sf L}'_{k+1}$ by definition.
Mapping $t$ down by $\eta^\iinfty_{2k}$, we end up in $\sD_{2k}'$ and hence the Sato-Levine maps vanishes.

Now consider $J^\iinfty\in\mathcal T^\iinfty_{2k}$. Then $\eta_{2k}(J^\iinfty)$ doubles $J$ to $\langle J,J\rangle$ and sums over putting a root at all of the leaves of one copy of $J$. The result represents naturally an element in $\sL_1\otimes \sL_{k+1}$.
Calculating the bracket has the effect of summing over putting a root near all of the leaves on one copy of $J$ in $\langle J,J\rangle$, which modulo IHX is equal to $(J,J)$. To see this requires pushing the central root of $(J,J)$ to one side using IHX relations.

Thus $sl_{2k}(\eta_{2k}(J^\iinfty))=1\otimes J$ which is equal to mapping $J^\iinfty$ right and then down.
\end{proof}

\section{The twisted Whitney tower classification}\label{sec:twisted-classification}
Using Theorem~\ref{thm:isomorphisms}, together with \cite{CST1,CST2}, we are ready to compute the groups ${\sf W}^\iinfty_n$.

\begin{proof}[Proof of Theorem~\ref{thm:R-013-isomorphisms}]
Consider the commutative diagram
$$\xymatrix{
\cT^\iinfty_n\ar@{->>}^{R_{n}^\iinfty}[r]\ar@{->>}_{\eta_n}[dr]&{\sf W}^\iinfty_n\ar@{->>}[d]^{\mu_n}\\
&{\sD}_n
}$$
where the realization map $R_n$ was defined and shown to be surjective in \cite{CST1}. The Milnor invariant $\mu_n$ was defined in \cite{CST2} where we also showed that the diagram commutes. By Theorem~\ref{thm:isomorphisms}, $\eta_n$ is an isomorphism for
$n\equiv 0,1,3\mod 4$ and the result follows.
\end{proof}

\begin{proof}[Proof of Theorem~\ref{thm:twistedalpha}]
It suffices to show there is an epimorphism 
\[
\alpha_k\colon\z\otimes{\sf L}_{k}\twoheadrightarrow{\sf K}^\iinfty_{4k-2}:=\Ker(\mu_{4k-2}\colon {\sf W}^\iinfty_{4k-2}\to \sD_{4k-2})
\]
This follows from the above commutative diagram, where $\eta_{4k-2}$ has the nontrivial kernel $\z\otimes\sL_{k}$, which therefore maps onto $\Ker(\mu_{4k-2})$. The fact that $\alpha_1$ is injective follows from Lemma~8 of \cite{CST2}.
\end{proof}

\section{The framed Whitney tower classification}\label{sec:framed-classification}
In this section we will compute the groups $\W_n$. The first observation is that,  mildly abusing notation, one can define Milnor invariants $\mu_n$ via the composition $\W_n\to\W^\iinfty_n\to\sD_n$. When $n$ is odd, this is the composition of two surjections, so it is surjective. However it is not surjective in the even case, and the computation of $\W_{2k}$ reflects that:
 
 \begin{thm}
The total Milnor invariant  $\mu_{2k}$ is a monomorphism with image $\sD'_{2k}<\sD_{2k}$, and the realization map 
$\widetilde{R}_{2k}\colon \cT_{2k}\to \W_{2k}$ is an isomorphism.
\end{thm}
 \begin{proof}
Using \cite{CST1,CST2} we again have commutative diagram of epimorphisms
 $$\xymatrix{
\cT_{2k}\ar@{->>}^{\widetilde{R}_{2k}}[r]\ar@{>->>}_{\eta'_{2k}}[dr]&{\sf W}_{2k}\ar[d]^{\mu_{2k}}\\
&{\sD}'_{2k}
}$$ 
By Theorem~\ref{thm:isomorphisms} (the Levine conjecture), $\eta'_{2k}$ is an isomorphism which proves our claim.
\end{proof}

\begin{proof}[Proof of Theorem \ref{thm:SL-invariants}]
The kernel ${\sf K}^\mu_{2k-1}$ of the Milnor invariant $\mu_{2k-1}:\W_{2k-1}\to \sD_{2k-1}$ sits in the middle of the above $4$-term exact sequence:
$$
\xymatrix{\W_{2k}\ar@{>->}[r]&\W^\iinfty_{2k}\ar[rr]\ar@{->>}[dr]&&\W_{2k-1}\ar@{->>}[r]&\W_{2k-1}^\iinfty\\
&&{\sf K}^\mu_{2k-1}\ar@{>->}[ur]&&
}
$$
Hence, the left-hand exact sequence gives rise to the following commutative diagram, which implies the existence of the dashed epimorphism on the right.
$$\xymatrix{
{\sf W}_{2k}\ar@{>->}[r]\ar[d]_\cong& {\sf W}^\iinfty_{2k}\ar@{->>}[r]\ar@{->>}[d]& {\sf K}^\mu_{2k-1}\ar@{-->>}[d]^{\SL_{2k-1}}\\
\sD'_{2k}\ar@{>->}[r]&\sD_{2k}\ar@{->>}[r]&\z\otimes \sL_{k+1}
}$$
When $k$ is even, the middle vertical map in the above diagram is an isomorphism, implying that $\SL_{2k-1}$ is an isomorphism as well.
\end{proof}
\begin{cor}
The realization maps $\widetilde{R}_{2k-1}$ are isomorphisms for even $k$.
\end{cor}
\begin{proof}
We use the exact sequence $0\to \z\otimes\sL'_{k+1}\to\widetilde{\cT}_{2k-1}\to \cT^\iinfty_{2k-1}\to 0$, and the fact that $\sL'_{k+1}=\sL_{k+1}$ for even $k$. This gives us a commutative diagram
$$
\xymatrix{
\z\otimes \sL_{k+1}\ar@{=}[d]\ar@{>->}[r]&\cT_{2k-1}\ar@{->>}^{\widetilde{R}_{2k-1}}[d]\ar@{->>}[r]&\cT^\iinfty_{2k-1}\ar[d]^\cong\\
\z\otimes \sL_{k+1}\ar@{>->}[r]&\W_{2k-1}\ar@{->>}[r]&\sD_{2k-1}
}
$$
which implies our claim by the 5-lemma.
\end{proof}

\begin{proof}[Proof of Proposition~\ref{prop:canonical-iso}]
The following commutative diagram explains why ${\sf K}^\iinfty_{4k-2}\cong {\sf K}^{\SL}_{4k-3}$.
$$\xymatrix{
&{\sf K}^\iinfty_{4k-2}\ar@{>-->>}[r]\ar@{>->}[d]&{\sf K}^{\SL}_{4k-3}\ar@{>->}[d]\\
{\sf W}_{4k-2}\ar@{>->}[r]\ar[d]_\cong& {\sf W}^\iinfty_{4k-2}\ar@{->>}[r]\ar@{->>}[d]& {\sf K}^\mu_{4k-3}\ar@{->>}[d]\\
\sD'_{4k-2}\ar@{>->}[r]&\sD_{4k-2}\ar@{->>}[r]&\z\otimes \sL_{2k}
}$$
\end{proof}

\section{Proof of Theorem~\ref{thm:Tau-sequences}}\label{sec:proof-tau-sequences}
The injectivity of the map $\cT_{2n}\hookrightarrow \cT^\iinfty_{2n}$ follows from Remark~\ref{rem:injective} and Corollary~\ref{cor:quadratic}, proven below.  The cokernel is then spanned by $\iinfty$-trees, with relations coming from the defining relations of $\cT^\iinfty_{2n}$, with non-$\iinfty$ trees set to $0$: $J^\iinfty=(-J)^\iinfty, I^\iinfty=H^\iinfty+X^\iinfty$, and $2J^\iinfty=0$. Thus the cokernel is isomorphic to $\z\otimes \sL'_{n+1}$.

The odd order sequence is shown to be exact as follows. Recall that the framing relations in $\widetilde{\cT}_{2n-1}$ are the image of $\Delta\colon \z\otimes \cT_{n-1}\to \cT_{2n-1}$. Thus we get an exact sequence as in the top of the diagram below, the middle exact sequence is Corollary 2.3 of \cite{L3}.
$$
\xymatrix{
(\z\otimes\cT_{n-1})/\Ker\Delta\ar@{>->}[r]^(.65){\Delta}\ar@{>->}[d]&\cT_{2n-1}\ar@{->>}[r]\ar[d]_\cong^{\eta'}&\widetilde{\cT}_{2n-1}\ar@{->>}[d]\\
\z^m\otimes\sL_n \ar@{=}[d] \ar@{>->}[r]^{sq} &{\sD}'_{2n-1} \ar@{>->}[d]\ar@{->>}[r]&{\sD}_{2n-1} \ar@{>->}[d]\\
\z\otimes\sL_1 \otimes\sL_n\ar@{>->}[r]^{sq} &\sL_1 \otimes\sL'_{2n}\ar@{->>}[r]&\sL_1 \otimes\sL_{2n}
}
$$
The map on the right is defined via the factorization $\widetilde{\cT}_{2n-1}\twoheadrightarrow {\cT}^\iinfty_{2n-1}\overset{\eta}{\longrightarrow} \sD_{2n-1}.$ So by definition the right-hand square commutes, and induces the left-hand vertical map. In fact, we claim that the induced map $\z\otimes\cT_{n-1}\to\z^m\otimes\sL_n$ factors as 
\[
\z\otimes\cT_{n-1}\overset{1\otimes\eta'}{\longrightarrow}\z\otimes{\sD}'_{n-1}\to\z^m\otimes\sL_n,
\]
 with the right hand map induced by $\sD'_{n-1}\to\Z^m\otimes\sL'_n\twoheadrightarrow\Z^m\otimes \sL_n$. To see this, let $t\in\cT_{n-1}$ and compute
\begin{eqnarray*}
\eta'\left(\Delta (1\otimes t)\right) &= &\eta'\left(\sum_{v\in t} \langle i(v),(T_v(t),T_v(t))\rangle\right)\\
&=&\sum_{v\in t} X_{i(v)}\otimes (T_v(t),T_v(t))\\
&=&\sum_{v\in t} X_i(v)\otimes sq(T_v(t))\\
&=&sq(1\otimes\eta')(1\otimes t)
\end{eqnarray*}
Now we claim that for all orders $n$, there is an exact sequence $$\z\otimes\sD'_{n-1}\to\z^m\otimes\sL_n\to\z\otimes \sL'_{n+1}\to 0.$$
This is clear if $n$ is odd, by tensoring the defining exact sequence for $\sD'_{n-1}$ with $\z$. If $n$ is even, then this follows since there is a surjection $\sD_{n-1}'\twoheadrightarrow \sD_{n-1}$ and $\sL'_{n+1} \cong\sL_{n+1}$. Therefore, the commutative diagram above supports a vertical short exact sequence on the left:
$$
\xymatrix{
&&\z\otimes\sL'_{n+1}\ar@{>-->}[d]\\
(\z\otimes\cT_{n-1})/\Ker\Delta\ar@{>->}[r]^(.65){\Delta}\ar@{>->}[d]&\cT_{2n-1}\ar@{->>}[r]\ar@{>->>}[d]^{\eta'}&\widetilde{\cT}_{2n-1}\ar@{->>}[d]\\
\z^m\otimes\sL_n\ar@{->>}[d]\ar@{>->}[r]&{\sD}'_{2n-1}\ar@{->>}[r]&{\sD}_{2n-1}\\
\z\otimes \sL'_{n+1}&&
}
$$
which gives us the indicated map on the right. Furthermore, we have already shown (without circularity) that $\eta\colon\cT^\iinfty_{2n-1}\to\sD_{2n-1}$ is an isomorphism, so the right-hand exact sequence is precisely the exact sequence we're interested in.


\section{Invariant forms and quadratic refinements}\label{sec:invt-forms-quadratic-refinements}
In this section we explain an algebraic framework into which our groups 
$\cT(m)$ and $\cT^\iinfty(m)$ fit naturally. 
In Lemma~\ref{lem:inner product} we show that the 
$\cT(m)$-valued inner product $\langle\,\cdot\,,\,\cdot\,\rangle$ on the free quasi-Lie algebra is universal. Then a general theory of quadratic refinements is developed and specialized from the non-commutative to the commutative to finally, symmetric settings. In Corollary~\ref{cor:quadratic} we show that $\cT_{2n}^\iinfty(m)$ is the home for the universal quadratic refinement of the $\cT_{2n}(m)$-valued inner product $\langle \ , \ \rangle$.

We work over the ground ring of integers but all our arguments go through for any commutative ring. We also only discuss the case of finite generating sets $\{1,\dots,m\}$, even though everything holds in the infinite case.

\subsection{A universal invariant form}
Via the identification of non-associative brackets with vertex-oriented labeled rooted trivalent trees explained in Section~\ref{sec:main-isos}, the free quasi-Lie algebra $\sL'(m)=\oplus_n \sL_n(m)$ on $m$ generators (Definition~\ref{def:eta-maps-and-D-tilde}) can be identified with the free abelian group on $\bL(m)$, modulo the AS and IHX relations. Here $\bL(m)$ denotes the set of (isomorphism classes of) rooted trees, trivalent and vertex-oriented as usual, with one specified root and the other univalent vertices labeled by elements in the set $\{1,\dots,m\}$.

Similarly, the group $\cT(m)=\oplus_n \cT_n(m)$ of Definition~\ref{def:tree-groups} is given by the free abelian group on $\bT(m)$, modulo the AS and IHX relations, where $\bT(m)$ is the set of (isomorphism classes of) unrooted trees, trivalent and vertex-oriented, with all univalent vertices labeled by elements in the set $\{1,\dots,m\}$. Gluing along the root gives the {\em inner product} from Section~\ref{sec:main-defs}:
\[
\langle \ , \ \rangle:\  \bL(m) \times \bL(m)\longrightarrow \bT(m)
\]
which is both {\em symmetric} and {\em invariant}: $\langle I,J \rangle = \langle J,I \rangle $ and 
\[
\langle (I,J),K \rangle = \langle I,(J,K) \rangle 
\]
where $(I,J)$ is the {\em rooted product} from Section~\ref{sec:main-defs} which turns into the {\em Lie bracket} $[I,J]$ when considered on $\sL'(m)$. This invariance follows by rotating the relevant planar tree by 120 degrees. 
The inner product extends uniquely to a bilinear, symmetric, invariant pairing
\[
\langle \ , \ \rangle:\ \sL'(m) \times \sL'(m) \longrightarrow \cT(m)
\]
This follows simply from observing that the AS and IHX relations hold on both sides and are preserved by the inner product. The following lemma shows that this inner product is {\em universal} for quasi-Lie algebras with $m$ generators. Here a quasi-Lie algebra is the same as a Lie algebra, except that the anti-symmetry relation takes the form $[X,Y] = -[Y,X]$ which is more appropriate when working over $\Z$ (or in characteristic 2). 
 
\begin{lem} \label{lem:inner product} Let $\g$ be a quasi-Lie algebra with a  bilinear, symmetric, invariant pairing $\lambda:\g \times \g \to M$ into some abelian group $M$. If $\alpha:\sL(m)\to\g$ is a quasi-Lie homomorphism (given by $m$ arbitrary elements in $\g$) there exists a {\em unique} linear map $\Psi:\cT(m)\to M$ such that  for all $X,Y\in \sL'(m)$
\[
\lambda(\alpha(X), \alpha(Y) ) = \Psi( \langle X,Y \rangle )
\]
\end{lem}
\begin{proof}
The uniqueness of $\Psi$ is clear since the inner product map is onto. For existence, we first construct a map $\psi:\bT(m)\to M$ as follows. Given a tree $t\in \bT(m)$ pick an edge in $t$ to split $t= \langle X,Y \rangle $ for rooted trees $X,Y\in \bL(m)$. Then set
\[
\psi(t) := \lambda(\alpha(X), \alpha(Y))
\]
If we split $t$ at an adjacent edge, this expression stays unchanged because of the symmetry and invariance of $\lambda$. However, one can go from any given edge to any other by a sequence of adjacent edges, showing that $\psi(t)$ does not depend on the choice of splitting.

It is clear that $\psi$ can be extended linearly to the free abelian group on $\bT(m)$ and since $\alpha$ preserves AS and IHX relations by assumption, this extension factors through a map $\Psi$ as required.
\end{proof}

\begin{rem} \label{rem:grading} Recall that the groups $\sL'(m)$ and $\cT(m)$ are graded by {\em order} and that the inner product preserves that grading (in the obvious sense) since the number of trivialent vertices just adds.
Moreover, $\sL'(m)[1]$ is actually a graded Lie algebra, i.e.\ the Lie bracket preserves the grading when shifted up by one (so order is replaced by the number of univalent non-root vertices).

Let's assume in the above lemma that the groups $\g,M$ are $\Z$-graded, $\g[1]$ is a graded Lie algebra and that $\lambda,\alpha$ preserve those gradings. Then the proof shows that the resulting linear map $\Psi$ also preserves the grading.
\end{rem}

\subsection{Non-commutative quadratic groups}
The rest of this section describes a general setting for relating our groups $\cT^\iinfty_{2n}$ that measure the intersection invariant of twisted Whitney towers to a universal (symmetric) quadratic refinement of the $\cT_{2n}$-valued inner product.
We first give a couple of definitions that generalize those introduced by Hans Baues in \cite{B1} and \cite[\S 8]{B2} and Ranicki in \cite[p.246]{R}. These will lead to the most general notion of quadratic refinements for which we construct a universal example. Later we shall specialize the definitions from {\em non-commutative} to {\em commutative} and finally, to {\em symmetric} quadratic forms and construct universal examples in all cases.
\begin{defn}\label{def:quadratic group} A {\em (non-commutative) quadratic group}
\[
\frak M = (M_{e} \overset{h}{\to} M_{ee} \overset{p}{\to} M_{e})
\]
consists of two groups $M_e, M_{ee}$ and two homomorphism $h,p$ satisfying 
\begin{enumerate}
\item $M_{ee}$ is abelian,
\item the image of $p$ lies in the center of $M_e$,
\item $hph=2h$. 
\end{enumerate}
$\frak M$ will serve as the range of the (non-commutative) quadratic forms defined below. We will write both groups additively since in most examples $M_e$ turns out to be commutative. A morphism $\beta: \frak M \to \frak{M}'$ between quadratic groups is a pair of homomorphisms 
\[
\beta_e: M_e\to M_e'  \quad \text{ and } \quad \beta_{ee}: M_{ee}\to M_{ee}'
\]
 such that both diagrams involving $h,h',p,p'$ commute.
\end{defn}

\begin{exs} \label{ex:M}
The example motivating the notation comes from homotopy theory, see e.g.\cite{B1}. For $m < 3n-2$, let $M_e= \pi_m(S^n)$, $M_{ee}=\pi_m(S^{2n-1})$, $h$ be the Hopf invariant and $p$ be given by post-composing with the Whitehead product $[\iota_n,\iota_n]:S^{2n-1}\to S^n$.

This quadratic group satisfies $php=2p$ which is part of the definition used in \cite{B1}, where $M_e$ is also assumed to be commutative.  As we shall see, these additional assumptions have the disadvantage that they are not satisfied for the universal example~\ref{ex:universal}. 

Another important example comes from an abelian group with involution $(M,*)$. Then we let
\begin{equation} \tag{$M$}
M_{ee}:=M, \quad M_e:=M/\langle x - x^* \rangle , \quad h([x]):=x+x^*
\end{equation}
and $p$ be the natural quotient map.  For example, $M$ could be a ring with involution, like a group ring $\Z[\pi_1X]$ with involution
\[
g^*= \pm w_1(g)\cdot g^{-1}
\]
 coming from an orientation character $w_1:\pi_1X\to\z$. 
\end{exs}

We note that in this example $hp-\id=*$ and in the homotopy theoretic example $hp-\id=(-1)^n$. In fact, we have the following
\begin{lem}\label{lem:involutions}
Given a quadratic group, the endomorphism $hp-\id$ gives an involution on $M_{ee}$ (which we will denote by $*$). Moreover, the formula $\dagger(x) := ph(x) -x$ defines an anti-involution on $M_e$. These satisfy
\begin{enumerate}
\item $*\circ h = h$, 
\item $php = p + p\circ *$,
\item $p \circ * = \dagger\circ p$. 
\end{enumerate}
\end{lem}
The proof of this Lemma is straightforward and will be left to the reader. To show that $\dagger$ is an anti-homomorphism one uses that $\im(p)$ is central and that $x\mapsto -x$ is an anti-homomorphism. 

\begin{defn}\label{def:quadratic group refinement}
A quadratic group $\frak M$ is a {\em quadratic refinement} of an abelian group with involution $(M,*)$ if
\[
M_{ee}=M  \quad \text{ and } \quad *=hp-\id.
\]
It follows from (i) in Lemma~\ref{lem:involutions} that in this case, the image of $h$ lies in the fixed point set of the involution: $h: M_e \to M^{\z} = H^0(\z;M)$. 
\end{defn}
The example $(M)$ above gives one natural choice of a quadratic refinement, however, there are other canonical (and non-commutative) ones as we shall see in Example~\ref{ex:universal}.

It follows from (ii) in Lemma~\ref{lem:involutions} that the additional condition $php=2p$ used in \cite{B1} is satisfied if and only if $p=p\circ *$, or equivalently, if $p$ factors through the cofixed point set of the involution:
\[
p: M_{ee} \sra (M_{ee})_{\z} = H_0(\z;M_{ee}) \to M_e
\]
It follows that the notion in \cite[p.246]{R} is equivalent to that in \cite{B1}, except that $M_{ee}$ is assumed to be the ground ring $R$ in the former. In that case, our involution is simply $r^* = \epsilon \bar r$, where $\epsilon=\pm 1$ and $r\mapsto \bar r$ is the given involution on the ring $R$. 

Then $\epsilon$-symmetric forms in the sense of Ranicki become hermitian forms in the sense defined below. In particular, Ranicki's $(+1)$-symmetric forms are different from the notion of {\em symmetric} form in this paper: We reserve it for the easiest case where both involutions, $*$ and $\dagger$, are trivial.

\subsection{Non-commutative quadratic forms}
\begin{defn}\label{def:quadratic} A {\em (non-commutative) quadratic form} on an abelian
group $A$ with values in a (non-commutative) quadratic group 
$\frak M= (M_{e} \overset{h}{\to} M_{ee} \overset{p}{\to} M_{e})$ is given by a bilinear map $\lambda:A\times A \to M_{ee}$ and a map $\mu:A\to M_{e}$ satisfying
\begin{enumerate}
\item $\mu(a+a')= \mu(a)+\mu(a')+p \circ \lambda(a,a')$ \ and
\item $h  \circ \mu(a) = \lambda(a,a) \ \forall a,a' \in A$.
\end{enumerate}
We say that $\mu$ is a {\em quadratic refinement} of $\lambda$: Property (i) says that $\mu$ is quadratic and property (ii) means that it `refines' $\lambda$. The notation  $M_e$ and $M_{ee}$ was designed (by Baues) to reflect the number of variables ({\bf e}ntries) of the maps $\mu$ and $\lambda$ respectively. He also writes $\lambda=\lambda_{ee}$ and $\mu=\lambda_e$, however, we decided not to follow that part of the notation.

We write $(\lambda,\mu):A\to \frak M$ for such quadratic forms and we always assume that the quadratic group $\frak M$ is part of the data for $(\lambda,\mu)$. 
This means that the morphisms in the category of quadratic forms are pairs of morphisms 
\[
\alpha:A\to A'  \quad \text{ and } \quad \beta=(\beta_e,\beta_{ee}): \frak M \to \frak{M}'
\]
such that both diagrams involving $\lambda, \lambda', \mu, \mu'$ commute.
\end{defn}
\begin{lem}\label{lem:symmetries}
Let $(\lambda,\mu):A\to \frak M$ be a quadratic form as above. Then
$\lambda$ is {\em hermitian} with respect to the involution $*=hp-\id$ on $M_{ee}$:
\[
\lambda(a',a) = \lambda(a,a')^*
\]
and 
 $\mu$ is {\em hermitian} with respect to the anti-involution $\dagger =ph -\id$ on $M_e$:
\[
\mu(-a) =  \mu(a)^\dagger
\]
\end{lem}
\begin{proof}
As a consequence of conditions (i) and (ii) we get
\begin{align*}
& \lambda(a,a)+ \lambda(a',a') + \lambda(a',a) + \lambda(a,a') = \lambda(a+a', a+a') =\\
& h\circ\mu(a+a') = h(\mu(a)+ \mu(a') +p\circ \lambda(a,a'))=\\
& \lambda(a,a)+ \lambda(a',a') + hp(\lambda(a,a'))
\end{align*}
or equivalently, $\lambda(a',a) = (hp-\id)\lambda(a,a') = \lambda(a,a')^*$. 
Similarly,
\begin{align*}
& 0=\mu(0) = \mu(a -a) = \mu(a) + \mu(-a) + p\circ \lambda(a,-a) =\\
& \mu(a) + \mu(-a) - p\circ h\circ \mu(a) =\mu(-a) + (\id - ph)\mu(a)
\end{align*}
or equivalently, $\mu(-a) = \dagger\circ \mu(a) =:\mu(a)^\dagger$.
\end{proof}
Starting with a hermitian form $\lambda$ with values in a group with involution $(M,*)$, the first step in finding a quadratic refinement for $\lambda$ is to find a quadratic refinement $\frak M$ of $(M,*)$ in the sense of Definition~\ref{def:quadratic group refinement}, motivating our terminology.

\subsection{Universal quadratic refinements}

\begin{ex}\label{ex:universal}
Given a hermitian form $\lambda:A \times A \to (M,*)$, one gets a quadratic refinement $\mu_\lambda$ of $\lambda$ as follows. Set $M_{ee}:=M$ and define the universal target $M_{e}:=M_{ee} \times_\lambda A$ to be the abelian group consisting of pairs $(m,a)$ with $m\in M_{ee}$ and $a\in A$ and multiplication given by
\[
(m,a) + (m',a') := (m + m' - \lambda(a,a'), a + a')
\]
In other words, $M_{e}$ is the central extension
\[
\xymatrix{
1 \ar[r]  & M_{ee} \ar[r]& M_{ee} \times_\lambda A  \ar[r]&  A \ar[r] & 1
}
\]
determined by the cocycle $\lambda$, compare Section~\ref{sec:presentations}. 
It follows that $M_e$ is commutative if and only if $\lambda$ is {\em symmetric} in the sense that $\lambda(a',a) = \lambda(a,a')$. Set
\[
 \quad p_\lambda(m):=(m,0),\quad  h_\lambda(m,a) :=  m +m^* + \lambda(a,a)
 \]
We claim that $\frak M_\lambda := (M_{ee}\overset{p_\lambda}{\to} M_e\overset{h_\lambda}{\to}M_{ee})$ is a quadratic group as in Definition~\ref{def:quadratic group}. It is clear that $p_\lambda$ is a homomorphism with image in the center of $M_e$. The homomorphism property of $h_\lambda$ follows from the fact that $\lambda$ is bilinear and hermitian:
\begin{align*}
& h_\lambda((m,a)+(m',a'))=
 h_\lambda(m+m'-\lambda(a,a'), a+ a') =\\
& (m+m'-\lambda(a,a')) + (m+m'-\lambda(a,a'))^* + \lambda(a+a',a+a')=\\
&(m+m^* + \lambda(a,a)) + (m'+ m'^* +\lambda(a',a'))=h_\lambda(m,a)+ h_\lambda(m',a')
\end{align*}
Condition (iii) of a quadratic group is also checked easily:
\begin{align*}
& h_\lambda p_\lambda h_\lambda(m,a) = h_\lambda(m+m^* + \lambda(a,a),0) =\\
& (m+m^*+\lambda(a,a)) + (m+m^*+\lambda(a,a))^* = \\
& 2(m+m^*+\lambda(a,a))= 2h_\lambda(m,a)
\end{align*}
We also see that 
\[
(h_\lambda p_\lambda -\id)(m) = h_\lambda(m,0) - m = (m+m^*) - m = m^*
\]
which means that $\frak M_\lambda$ ``refines'' the group with involution $(M,*)$. 
Finally,  setting $\mu_\lambda(a):= (0,a)$,  we claim that $(\lambda,\mu_\lambda):A\to \frak M_\lambda$ is a quadratic refinement of $\lambda$. We need to check properties (i) and (ii) of a quadratic form:
(i) is simply $h_\lambda \circ \mu_\lambda(a) = h_\lambda(0,a) = \lambda(a,a)$
and (ii) explains why we used a sign in front of $\lambda$ in our central extension:
\begin{align*}
& \mu_\lambda(a)+\mu_\lambda(a') + p_\lambda \circ\lambda(a,a') =
(0,a) + (0,a') + (\lambda(a,a'),0) =\\
& (-\lambda(a,a'),a+a') + (\lambda(a,a'), 0) = (0,a+a') = \mu_\lambda(a+a')
\end{align*}
\end{ex}

The following result will show that $\mu_\lambda$ is indeed a {\em universal} quadratic refinement of $\lambda$. This is the content of the first statement in the theorem below. It follows from the second statement because for any quadratic refinement $\mu$ of $\lambda$ it shows that forgetting the quadratic data gives canonical isomorphisms
\[
\QF(L\circ R(\lambda,\mu), (\lambda,\mu)) \cong \HF( R(\lambda,\mu) , R(\lambda,\mu)) = \HF(\lambda,\lambda)
\]
where $\QF$ respectively $\HF$ are (the morphisms in) the categories of quadratic respectively hermitian forms. Since 
\[
L\circ R(\lambda,\mu) = L(\lambda)= (\lambda,\mu_\lambda)
\]
 and the morphisms in the category $ \QR_\lambda$ of quadratic refinements of $\lambda$ 
 by definition all lie over the identity of $\lambda$, the set $ \QR_\lambda(\mu_\lambda,\mu)$
contains a unique element, namely the required universal morphism $\mu_\lambda \to \mu$.

\begin{thm}\label{thm:universal}
The quadratic form $(\lambda,\mu_\lambda)$ is initial in the category of quadratic refinements of $\lambda$. In fact, the forgetful functor $R(\lambda,\mu)=\lambda$ from the category of quadratic forms to the category of hermitian forms has a left adjoint $L:\HF\to \QF$ given by $L(\lambda):= (\lambda,\mu_\lambda)$.
\end{thm}

\begin{proof}
We have to construct natural isomorphisms
\[
\QF((\lambda,\mu_{\lambda}), (\lambda',\mu')) = \QF(L(\lambda), (\lambda',\mu')) \cong \HF( \lambda , R(\lambda',\mu')) = \HF(\lambda, \lambda')
\]
for any quadratic form $(\lambda',\mu')$ and hermitian form $\lambda$. Recall that the morphisms in $\QF$ are pairs $\alpha: A\to A'$ and $\beta=(\beta_e,\beta_{ee}): \frak M \to \frak M'$ such that the relevant diagrams commute. This implies that forgetting about the quadratic datum $\beta_e$ gives a natural map from the left to the right hand side above. 

Given a morphism $(\alpha,\beta_{ee}): \lambda\to \lambda'$ consisting of homomorphism $\alpha:A\to A'$ and $\beta_{ee}: (M_{ee},*) \to (M_{ee}',*')$ such that
\[
\lambda'(\alpha(a_1), \alpha(a_2)) = \beta_{ee} \circ \lambda(a_1, a_2) \in M_{ee}' \quad \forall\ a_i\in A
\]
we need to show that there is a {\em unique} homomorphism $\beta_e: M_e \to M_e'$ such that the following 3 diagrams commute
\[
\xymatrix{
\ar @{} [dr] |{(1)} M_{e} \ar[r]^{h} \ar[d]_{\beta_e} & M_{ee} \ar[d]^{\beta_{ee}} 
& \ar @{} [dr] |{(2)} M_{ee} \ar[r]^{p} \ar[d]_{\beta_{ee}} & M_{e} \ar[d]^{\beta_{e}} 
& \ar @{} [dr] |{(3)} A \ar[r]^{\mu_{\lambda}} \ar[d]_{\alpha} & M_{e} \ar[d]^{\beta_{e}} \\
 M_{e}' \ar[r]^{h'} & M_{ee}'
 &  M_{ee}' \ar[r]^{p'} & M_{e}'  
 &  A' \ar[r]^{\mu'} & M_{e}'  
}
\]
We will now make use of the fact that $M_e=M_{ee} \times_\lambda A$ because $\mu_\lambda$ is given as in Example~\ref{ex:universal}. In this case, diagrams (2) and (3) are equivalent to
\[
\beta_e(m,0) = p'\circ\beta_{ee}(m)  \quad \text{ and } \quad \beta_e(0,a) = \mu'\circ \alpha(a)
\]
because $p(m) = (m,0)$ and $\mu_\lambda(a) = (0,a)$. 
This implies directly the uniqueness of $\beta_e$. For existence, we only have to check that the formula
\[
\beta_e(m,a) := p'\circ\beta_{ee}(m) + \mu'\circ \alpha(a)
\]
gives indeed a group homomorphism $M_e \to M_e'$ that makes diagram (1) commute. Note that the image of $p'$ is central in $M_e'$ and hence the order of the summands does not matter. We have
\begin{align*}
& \beta_e((m,a)+(m',a'))=
 \beta_e(m+m'-\lambda(a,a'), a+ a') &=\\
& p'\circ\beta_{ee}(m+m'-\lambda(a,a')) + \mu'\circ \alpha(a+a')&=\\
&p'\circ\beta_{ee}(m) + p'\circ\beta_{ee}(m') - p'\circ \lambda'(\alpha(a),\alpha(a')) + \mu'\circ \alpha(a+a')&=\\
&p'\circ\beta_{ee}(m)  + p'\circ\beta_{ee}(m') + \mu'\circ \alpha(a) + \mu'\circ \alpha(a')&=\\
&\beta_e(m,a)+ \beta_e(m',a')
\end{align*}
To get to the forth line,  we used property (ii) of a quadratic form to cancel the term $p'\circ \lambda'(\alpha(a),\alpha(a'))$. For the commutativity of diagram (1) we use property (i) of a quadratic form, as well as the fact that $\beta_{ee}$ preserves the involution $*$:
\begin{align*}
&h'\circ \beta_e(m,a) = h' (p'\circ \beta_{ee}(m) + \mu'\circ\alpha(a))=\\
&h'p'(\beta_{ee}(m)) + \lambda'(\alpha(a),\alpha(a)) = \beta_{ee}(m)^{*'} + \beta_{ee}(m) + \beta_{ee} \circ\lambda(a,a) =\\
 & \beta_{ee}(m^*+ m +\lambda(a,a)) = \beta_{ee}\circ h(m,a) 
\end{align*}
This finishes the proof of left adjointness of $L:\HF\to \QF$.
\end{proof}
If the bilinear form $\lambda$ happens to be  {\em symmetric}, or more precisely, if it takes values in a group $M_{ee}$ with {\em trivial} involution $*$, then the above construction still gives a quadratic refinement $\mu_\lambda$. Its target quadratic group $\frak M_\lambda$ has the properties that $M_e$ is abelian and $h_\lambda p_\lambda = 2\id$. 
It is not hard to see that our construction above leads to the following result.

\begin{thm}\label{thm:universal}
For any symmetric form $\lambda$ one can functorially construct a quadratic form $(\lambda,\mu_\lambda)$ that is initial in the category of quadratic refinements of $\lambda$ with trivial involution $*$. In fact, the forgetful functor $R(\lambda,\mu)=\lambda$ from the category of quadratic forms with trivial involution $*$ to the category of symmetric forms has a left adjoint $L(\lambda)= (\lambda,\mu_\lambda)$.
\end{thm}

\begin{rem}
It follows from the above considerations that a quadratic form $(\lambda,\mu)$ is universal if and only if the homomorphism
\[
M_{ee} \times_\lambda A \to M_e  \quad \text{ given by } \quad (m,a)\mapsto p(m) + \mu(a)
\]
is an isomorphism. This is turn is equivalent to
\begin{enumerate}
\item $p:M_{ee} \to M_e$ is injective and
\item $\mu : A \to M_e/\im(p)$ is an isomorphism.
\end{enumerate}
\end{rem}

\subsection{Commutative quadratic groups and forms}
The case where $*$ is non-trivial but the anti-involution $\dagger$ on $M_e$ is trivial is even more interesting. In this case, $\lambda$ is still hermitian with respect to $*$ but one is only interested in quadratic refinements $\mu$ that are symmetric in the sense that $\mu(-a)=\mu(a)$. This case deserves its own definition:

\begin{defn}\label{def:commutative quadratic group} A {\em commutative quadratic group}
\[
\frak M = (M_{e} \overset{h}{\to} M_{ee} \overset{p}{\to} M_{e})
\]
consists of two abelian groups $M_e, M_{ee}$ and two homomorphism $h,p$ satisfying $ph=2\id$. 
\end{defn}

In fact, a commutative quadratic group is the same thing as a non-commutative quadratic group with trivial anti-involution $\dagger$. This comes from the fact that the squaring map $x\mapsto 2x$ is a homomorphism if and only if $M_e$ is commutative. Our universal example $\frak M_\lambda$ is in general {\em not} commutative because one gets in this case
\begin{align*}
& \dagger_\lambda(m,a) = p_\lambda \circ h_\lambda(m,a) - (m,a) = p_\lambda(m+m^* +\lambda(a,a)) - (m,a) =\\
& (m+m^* +\lambda(a,a),0) + (-m-\lambda(a,a),-a) = (m^*, -a)
\end{align*}
However, we shall see in Theorem~\ref{thm:commutative universal} that we can just divide by these relations $(m,a) = (m^*,-a)$ to obtain another universal quadratic refinement of a given hermitian form $\lambda$ but this time with values in a {\em commutative} quadratic group. Before we work this out, let us mention the essential example from topology.
\begin{ex}
Consider a manifold $X$ of dimension~$2n$ and let $\frak M$ be as in (M) from Example~\ref{ex:M} with $M=\Z[\pi_1X]$.  In particular, we have $ph-\id=\dagger=\id$ but in general the involution $*$ is non-trivial. On group elements, it is given by
\[
g^* := (-1)^n w_1(g) g^{-1}
\]
Then the equivariant intersection form $\lambda=\lambda_X$ on $A=\pi_nX$ is bilinear and hermitian as required. Moreover, the self-intersection invariant $\mu_X$ defined by Wall \cite{Wa} gives a quadratic refinement of $\lambda_X$, at least on the subgroup of $A$ represented by immersed $n$-spheres with stably trivial normal bundle. 
\end{ex}

In our main Theorem below, we shall use the following
\begin{lem}\label{lem:square}
If $(\lambda,\mu):A\to \frak M$ is a commutative quadratic form, then $\mu(n\cdot a)=n^2\cdot \mu(a)$ for all integers $n\in \Z$.
\end{lem}
Here we say that a quadratic form $(\lambda,\mu):A\to \frak M$ is {\em commutative} if the target quadratic group $\frak M$ is commutative, i.e.\ if the anti-involution $\dagger$ is trivial. 
\begin{proof}
Since the involution $\dagger=ph-\id$ is trivial by assumption, we already know that $\mu(-a)=\mu(a)$ from Lemma~\ref{lem:symmetries}. Thus it suffices to prove the claim for positive $n>1$ by induction:
\begin{align*}
\mu((n+1)\cdot a) &= \mu(n\cdot a) + \mu(a) + p\circ\lambda(n\cdot a,a) \\
&= n^2\cdot \mu(a) + \mu(a) + n\cdot p\circ h\circ \mu(a) \\
&= (n^2+1)\cdot \mu(a) + n \cdot 2\cdot \mu(a) = (n+1)^2\cdot \mu(a)
\end{align*}
Here we again used the fact that $p\circ h=2\id$.
\end{proof}

\begin{thm}\label{thm:commutative universal}
Any hermitian bilinear form $\lambda$ has a universal commutative quadratic refinement. In fact, the forgetful functor $R(\lambda,\mu)=\lambda$ from the category $\CQF$ of commutative quadratic forms to the category $\HF$ of hermitian forms has a left adjoint $L:\HF\to\CQF, L(\lambda)=(\lambda,\mu^c_\lambda)$.
\end{thm}

\begin{proof}
As hinted to above, we will force the anti-involution $t$ to be trivial in the universal construction of Theorem~\ref{thm:universal}. This means that we should define the universal (commutative) group $M^c_e$ as the quotient of our previously used group $M_{ee} \times_\lambda A$ by the relations
\begin{align*}
0 &= (m^*, -a) - (m,a) = (m^*,-a) + (-m-\lambda(a,a)),-a) \\
& =(m^* - m - 2\lambda(a,a), -2a)
\end{align*}
By setting $a$ respectively $m$ to zero, these relations imply
\[
(m^*,0) = (m,0)  \quad \text{ and } \quad (-2\lambda(a,a), -2a) = 0
\]
Vice versa, these two types of equations imply the general ones and hence we see that $M^c_e$ is the quotient of the centrally extended group
\[
\xymatrix{
1 \ar[r]  & M_{ee}/(m^*=m) \ar[r]& M_{ee}/(m^*=m) \times_\lambda A  \ar[r]&  A \ar[r] & 1
}
\]
by the relations $(-2\lambda(a,a), -2a) = 0$. We write elements in $M^c_e$ as $[m,a]$ with the above relations understood. It then follows that $p^c_\lambda(m):= [m,0]$ is a homomorphism $M_{ee}\to M^c_e$ (which is in general not any more injective). Moreover, our original formula leads to a homomorphism $h^c_\lambda:M^c_e\to M_{ee}$ given by
\[
h^c_\lambda[m,a]:= h_\lambda(m,a)= m+m^* + \lambda(a,a)
\]
To see that this is well defined, observe $h_\lambda(m^*,0) = m+m^* = h_\lambda(m,0)$ and
\[
h_\lambda(-2\lambda(a,a),-2a)= -4\lambda(a,a) + \lambda(-2a,-2a) =0 
\]
Finally, we set $\mu^c_\lambda(a):=[0,a]$ to obtain a commutative quadratic refinement of $\lambda$ which is proven exactly as in Theorem~\ref{thm:universal}. 

To show that $\mu^c_\lambda$ is universal, or more generally, that $L(\lambda):=(\mu^c_\lambda,\lambda)$ is a left adjoint of the forgetful functor $R$, we proceed as in the proof of Theorem~\ref{thm:universal}. 
So we are given a morphism $(\alpha,\beta_{ee}): \lambda\to \lambda'$ consisting of homomorphism $\alpha:A\to A'$ and $\beta_{ee}: (M_{ee},*) \to (M_{ee}',*')$ such that
\[
\lambda'(\alpha(a_1), \alpha(a_2)) = \beta_{ee} \circ \lambda(a_1, a_2) \in M_{ee}' \quad \forall\ a_i\in A
\]
we need to show that there is a {\em unique} homomorphism $\beta_e: M^c_e \to M_e'$ such that the 3 diagrams from the proof of Theorem~\ref{thm:universal} commute.
We can use the same formulas as before, if we check that they vanish on our new relations in $M^c_e$. For this we'll have to use that the given quadratic group $\frak M'$ is {\em commutative}. Recall the formula
\[
\beta_e(m,a) = p'\circ\beta_{ee}(m) + \mu'\circ \alpha(a)
\]
Splitting our relations into two parts as above, it suffices to show that
\[
p'\circ\beta_{ee}(m^*) = p'\circ\beta_{ee}(m)  \quad \text{ and } \quad \beta_e(-2\lambda(a,a),-2a)=0
\]
The first equation follows from part (iii) of Lemma~\ref{lem:involutions} and the fact that we are assuming that $\dagger'=\id$:
\begin{align*}
 p'\circ\beta_{ee}(m^*) = (p'\circ *')(\beta_{ee}(m)) =  (\dagger' \circ p')(\beta_{ee}(m)) = p'\circ\beta_{ee}(m)
\end{align*}
For the second equation we compute:
\begin{align*}
& \beta_e(-2\lambda(a,a),-2a)=p'\circ\beta_{ee}(-2\lambda(a,a)) + \mu'\circ \alpha(-2a)=\\
& -2(p' \circ\lambda'(\alpha(a),\alpha(a))) + \mu'\circ \alpha(-2a)=\\
 &-2(\mu'(\alpha(a) + \alpha(a)) - \mu'(\alpha(a)) - \mu'(\alpha(a))) + \mu'(-2\alpha(a))=\\
 &-2(4\mu'(\alpha(a)) - 2\mu'(\alpha(a))) + 4\mu'(\alpha(a))= -4\mu'(\alpha(a)) + 4\mu'(\alpha(a))=0
\end{align*}
We used Lemma~\ref{lem:square} for $n=\pm 2$ and hence the commutativity of $\frak M$.
\end{proof}

\subsection{Symmetric quadratic groups and forms}

The simplest case of a quadratic group is where both $*$ and $\dagger$ are trivial. Let's call such a quadratic group $\frak M= (M_{e} \overset{h}{\to} M_{ee} \overset{p}{\to} M_{e})$ {\em symmetric}. Equivalently, this means that $hp=2\id=ph$ (and hence $M_e$ is commutative). Then a quadratic form $(\lambda,\mu):A\to\frak M$ will automatically be {\em symmetric} in the sense that
\[
\lambda(a,a') = \lambda(a',a)  \quad \text{ and } \quad \mu(-a) = \mu(a)\quad \forall \ a\in A.
\]
We call $\mu$ a {\em symmetric quadratic refinement} of $\lambda$ and obtain a category of symmetric quadratic forms with a forgetful functor $R$ to the category of symmetric forms. It is not hard to show that the construction in Theorem~\ref{thm:commutative universal} gives a universal symmetric quadratic refinement $\mu^c_\lambda$ for any given symmetric bilinear form $\lambda$. More precisely, 

\begin{thm}\label{thm:symmetric universal}
Any symmetric bilinear form $\lambda$ has a universal symmetric quadratic refinement. In fact, the forgetful functor $R(\lambda,\mu)=\lambda$ from the category $\SQF$ of symmetric quadratic forms to the category $\SF$ of symmetric forms has a left adjoint $L:\HF\to\CQF, L(\lambda)=(\lambda,\mu^c_\lambda)$.
\end{thm}

\begin{rem}\label{rem:injective}
We observe that the map $p^c_\lambda: M_{ee}\to M^c_e$ is a monomorphism in this easiest, symmetric, case, just like it was in the hardest, non-commutative, case. This can be seen by noting that the first set of relations $(m^*,0)=(m,0)$ is redundant if the involution $*$ is trivial. Therefore, if $0=p^c_\lambda(m) = [m,0]$ then $(m,0)$ must come from the second set of relations, i.e.\ it must be of the form 
\[
(m,0)=(-2 \lambda(a,a), -2a)  \quad \text{ for some } a\in A.
\]
 This implies that $2a=0$ and hence $\lambda(2a,a)=0$ which in turn means $m=0$.
\end{rem}

\begin{exs}\label{ex:quadratic}
If $M_{ee}=M_{e}$ then $p=\id$ and $h =2\id$ are canonical choices.  If $A$ is free then a quadratic refinement with this choice of $(M_{e}, h, p)$ exists exactly for even forms. Moreover, if $M_{ee}$ has no 2-torsion than a quadratic refinement is uniquely determined by the given even form. 

At the other extreme, consider $M_{ee}=M_{e}=\z$. If $A$ is a finite dimensional $\z$-vectorspace then non-singular symmetric bilinear forms $\lambda$ are classified by their rank and their {\em parity}, i.e.~whether they are even or odd, or equivalently, whether they admit a quadratic refinement or not. In the even case, quadratic forms $(\lambda,\mu)$ are classified by rank and {\em Arf invariant}. This additional invariant takes values in $\z$ and vanishes if and only if $\mu$ takes more elements to zero than to one (thus the Arf invariant is sometimes referred to as the ``democratic invariant''). 

If $\lambda$ is odd then the following trick allows to still define Arf invariants and it motivates the introduction of $M_{e}$. Let again $A$ be a finite dimensional $\z$-vectorspace, $M_{ee}=\z$ and $M_{e}=\Z_4$ with the unique nontrivial homomorphisms $h, p$. Then any non-singular symmetric bilinear form $\lambda$ has a quadratic refinement $\mu$ and quadratic forms $(\lambda,\mu)$ are classified by rank and an Arf invariant with values in $\Z_8$. If $\lambda$ is even, this agrees with the previous Arf invariant via the linear inclusion $\z \subset \Z_8$.
\end{exs}

\subsection{Presentations for universal quadratic groups} \label{sec:presentations}
Consider a central extension of groups
\[
1 \to M \to G \overset{\pi}{\to} A \to 1
\]
and assume that $M$ and $A$ have presentations $ \langle m_i | n_j \rangle$ respectively $ \langle a_k | b_\ell \rangle $. To avoid confusion, we write groups multiplicatively for a while and switch back to additive notation when returning to hermitian forms. 

It is well known how to get a presentation for $G$: Pick a section $s:A\to G$ with $s(1)=1$ which is not necessary multiplicative. Write a relation in $A$ as $b_\ell = a'_1 \cdots a'_r$, where $a'_i$ are generators of $A$ or their inverses, then
\[
1=s(1)=s(b_\ell)=s(a'_1)\cdots s(a'_r) \, w_\ell
\]
where $w_\ell=w_\ell(m_i)$ is a word in the generators of $M$. This equation follows from the fact that the projection $\pi$ is a homomorphism and for simplicity we have identified $M$ with its image in $G$. We obtain the presentation
\[
G = \langle m_i, \alpha_k\, |\, n_j, [m_i,\alpha_k], \beta_\ell \, w_\ell \rangle 
\]
where $\alpha_k:=s(a_k)$ and $\beta_\ell:=s(a'_1)\cdots s(a'_r)$ is the same word in the $\alpha_k$ as $b_\ell$ is in the $a_k$. The commutators $[m_j,\alpha_k]$ arise because we are assuming that the extension is central, in a more general case one would write out the action of $A$ on $M$. 

It will be useful to rewrite this presentation as follows. Observe that the section $s$ satisfies
\[
s(a_1a_2)=s(a_1)s(a_2) c(a_1,a_2)
\]
for a uniquely determined {\em cocycle} $c: A \times A\to M$. By induction one shows that
\begin{align*}
& s(a_1\cdots a_r) = s(a_1\cdots a_{r-1})s(a_r) c(a_1\cdots a_{r-1},a_r)= \cdots = \\
&s(a_1)\cdots s(a_r) c(a_1, a_2) c(a_1a_2,a_3)c(a_1a_2a_3,a_4)\cdots c(a_1\cdots a_{r-1},a_r)
\end{align*}
Comparing this expression with the definition of the word $w_\ell$ in the presentation of $G$, it follows that 
\[
w_\ell= c(a'_1, a'_2) c(a'_1a'_2,a'_3)\cdots c(a'_1\cdots a'_{r-1},a'_r) \ \in M
\]
so that the above presentation of $G$ is entirely expressed in terms of the cocycle $c$ (and does not depend on the section $s$ any more). 

Now assume that $\lambda:A \times A\to M$ is a hermitian form with respect to an involution $*$ on $M$. Then the universal (non-commutative) quadratic group $M_e$ from Example~\ref{ex:universal} is a central extension as above with cocycle $c=\lambda$. Reverting to additive notation, we see that 
\begin{align*}
w_\ell&= \lambda(a'_1, a'_2) + \lambda(a'_1+a'_2,a'_3)+ \cdots + \lambda(a'_1+\cdots +a'_{r-1},a'_r)\\
&= \sum_{1\leq i < j \leq r} \lambda(a'_i,a'_j)
\end{align*}
where the ordering of the summands is irrelent because $M$ is central in $M_e$. Summarizing the above discussion, we get.
\begin{lem} \label{lem:presentation}
The universal (non-commutative) quadratic group $M_e$ corresponding to the hermitian form $\lambda$ has a presentation
\[
M_e= \langle m_i, \alpha_k\, |\, n_j, [m_i,\alpha_k], \beta_\ell + \sum_{1\leq i < j \leq r} \lambda(a'_i,a'_j) \rangle 
\]
where the generators $m_i, \alpha_k$ and words $n_j,\beta_\ell$ are defined as above. Moreover, the universal quadratic refinement $\mu: A\to M_e$ is a section of the central extension and hence $\alpha_k = \mu(a_k)$ for the generators $a_k$ of $A$.
\end{lem}
As discussed in Theorem~\ref{thm:commutative universal}, we get the universal {\em commutative} quadratic group $M_e^c$ for $\lambda$ by adding the relations $(m^*,0) = (m,0)$ and $(-2\lambda(a,a),-2a)=0$. The latter can be rewritten in the form $2(0,a)=(\lambda(a,a),0)$. In the current notation, where $(m,0)$ is identified with $m\in M$, we obtain the relations
\[
m^* = m  \quad \text{and} \quad 2 \mu(a)= \lambda(a,a)  \ \in M_e^c \quad\forall\ m\in M, a\in A.
\]
Recallling that $A,M$ and $M_e^c$ are {\em commutative} groups, we can write our presentation in that category
to obtain
\begin{lem} \label{lem:commutative presentation}
The universal (commutative) quadratic group $M^c_e$ corresponding to the hermitian form $\lambda$ has a presentation
\[
M^c_e= \langle m_i, \mu(a_k)\, |\, n_j, \beta_\ell + \sum_{1\leq i < j \leq r} \lambda(a'_i,a'_j) , m^* = m  ,2 \mu(a)= \lambda(a,a) \ \rangle 
\]
Here $a_k$ are generators of $A$ and for every relation $b_\ell = \sum_{i=1}^r a'_i$ in $A$, we defined the word $\beta_\ell:= \sum_{i=1}^r \mu(a'_i)$.
\end{lem}

\subsection{Twisted intersection invariants and a universal quadratic group}

If we apply the universal construction to the symmetric bilinear form on order $n$ rooted trees
\[
\langle \ , \ \rangle:\ \sL'_{n+1} \times \sL'_{n+1} \longrightarrow \cT_{2n}=:(\cT_{2n})_{ee}
\]
we obtain a universal symmetric quadratic refinement 
\[
q:=\mu^c_{\langle \ , \ \rangle}: \sL'_{n+1}\to (\cT_{2n})^c_e
\]
Let us apply the presentation from Lemma~\ref{lem:commutative presentation} to this case.
The generators of $\sL'_{n+1}$ are oriented rooted trees $J$ of order $n$ and the relations are the AS and IHX relations: $J +\bar J=0$ if $\bar J$ is $J$ with reversed orientation, and the usual local relation $I + H +X=0$. Similarly, $\cT_{2n}$ is generated by  oriented un-rooted trees $t$ of order $2n$, modulo the AS and IHX relations. Putting these together, we see that $(\cT_{2n})^c_e$ is generated by trees $t$ and $q(J)$ and the relations are
\begin{itemize}
\item[$n_j:$] Ordinary AS and IHX relations for unrooted trees $t$,
\item[$\beta_\ell:$] Twisted AS and IHX relations for rooted trees $J$: 
\begin{align*}
&q(J) + q(\bar J) + \langle J, \bar J \rangle =0\\
&q(I) + q(H) + q(X) + \langle I,H \rangle + \langle I,X \rangle + \langle H,X \rangle =0
\end{align*}
\item[$c:$] $2q(J) = \langle J, J \rangle $
\end{itemize}
The last relation $c$ builds in the commutativity of the universal group as discussed above because we are in the easiest, symmetric, setting where the involution $*$ is trivial. Using relation $c$, the twisted AS relation simply becomes
\[
q(\bar J) = q(-J) = q(J)
\]
which was expected since we are in the symmetric case. 

The reader may be surprised to discover that this presentation is exactly the home $ \cT^\iinfty_{2n}$ for the intersection invariants $\tau_{2n}^\iinfty$ of order $2n$ twisted Whitney towers \cite{CST1}. 

\begin{cor} \label{cor:quadratic} There is an isomorphism of symmetric quadratic groups
\[
(\cT_{2n})^c_e \cong  \cT^\iinfty_{2n}
\]
which is the identity on $ \cT_{2n}$ and takes $q(J)$ to $J^\iinfty$ for rooted trees $J$. 
The quadratic group structure on  $\cT^\iinfty_{2n}$ is given by the homomorphisms $\cT_{2n}\overset{p}{\to} \cT^\iinfty_{2n}\overset{h}{\to}\cT_{2n}$ which are uniquely characterized (for unrooted trees $t$ and rooted trees $J$) by
\[
p(t)=t  \quad \text{and} \quad h(t) = 2\cdot t, \ h(J^\iinfty) = \langle J,J \rangle 
\]
\end{cor}


\end{document}